\documentclass[a4paper, reqno, 11pt]{amsart}

\usepackage[english]{babel}
\usepackage{amsmath}
\usepackage{amssymb}
\usepackage{amsthm}
\usepackage{mathrsfs}
\usepackage{enumerate}
\usepackage{ifthen}
\usepackage{bbm}
\usepackage{xcolor}
\usepackage{verbatim}

\usepackage[pdftex,     
plainpages=false,   
breaklinks=true,    
colorlinks=true,
linkcolor=blue,
citecolor=green,
pdftitle={Title},
pdfauthor={P. Bonicatto, G. Ciampa and G. Crippa}
]{hyperref}

\usepackage{graphicx}
\usepackage{pgfplots}
\usepackage{tikz} 
\usepackage{tikz-3dplot}
\usetikzlibrary{patterns, patterns.meta}
\usepackage{float}
\usepackage{mathtools}
\usepackage[font=footnotesize,labelfont=bf]{caption}
\usepackage{xfrac}
\usepackage{enumitem}
\usepackage{bm}
\usepackage{bbm}
\renewcommand{\b}{\bm b}
\renewcommand{\a}{\bm a}

\newcommand{\bsigma}{\bm \sigma}

\newcommand{\margnote}[1]{
\ifthenelse{\boolean{shownotes}}%
{\marginpar{\raggedright\tiny\texttt{#1}}}%
{}%
}

\newcommand{\hole}[1]{
\ifthenelse{\boolean{shownotes}}%
{\begin{center} \fbox{ \rule {.25cm}{0cm}
\rule[-.1cm]{0cm}{.4cm} \parbox{.85\textwidth}{\begin{center}
\texttt{#1}\end{center}} \rule {.25cm}{0cm}}\end{center}}
{}
}
\newtheorem{thm}{Theorem}[section]
\newtheorem{prop}[thm]{Proposition}
\newtheorem{proposition}[thm]{Proposition}
\newtheorem{lem}[thm]{Lemma}
\newtheorem{cor}[thm]{Corollary}
\newtheorem{rem}[thm]{Remark}

\theoremstyle{definition}
\newtheorem{defn}[thm]{Definition}
\newtheorem{definition}[thm]{Definition}

\newtheorem*{mainthm*}{Selection Theorem}
\newtheorem*{mainthmdue*}{Regularity Theorem}

\DeclareMathOperator{\dive}{\mathrm {div}}

\newcommand{\e}{\varepsilon}		       
\newcommand{\R}{\mathbb{R}}
\newcommand{\T}{\mathbb{T}}

\newcommand{\weakto}{\rightharpoonup}
\newcommand{\weaktos}{\stackrel{*}{\rightharpoonup}}

\newcommand{\de}{\mathrm{d}}
\newcommand*{\paolo}[1]{\textcolor{red}{#1}}
\newcommand*{\gennaro}[1]{\textcolor{blue}{#1}}

\usepackage{float}   
\usepackage{fancybox}		  
\usepackage{verbatim}

\newcommand{\schema}[1]{{\bf \sc #1}}
\newenvironment{pschema}[1]{\vspace{3mm}
\noindent\begin{Sbox}\begin{minipage}{.95\columnwidth}\vspace{2mm}\begin{center}{\large \schema{#1}}\vspace{5mm}\\
\begin{minipage}{0.9\textwidth}}{\end{minipage}\end{center}\vspace{2mm}\end{minipage}\end{Sbox}\fbox{\TheSbox}\vspace{3mm}}

\numberwithin{equation}{section}

\title[Regularity and uniqueness for Fokker-Planck equations]{A regularity result for the Fokker-Planck equation with non-smooth drift and diffusion}

\author{Paolo Bonicatto}
\address[P.\ Bonicatto]{Dipartimento di Matematica, Universit\`a di Trento, Via Sommarive, 14, 38123 Povo (TN), Italy.}
\email{paolo.bonicatto@unitn.it}

\author[Gennaro Ciampa]{Gennaro Ciampa}
\address[G.\ Ciampa]{DISIM - Dipartimento di Ingegneria e Scienze dell'Informazione e Matematica\\ Universit\`a  degli Studi dell'Aquila \\Via Vetoio \\ 67100 L'Aquila \\ Italy}
\email{gennaro.ciampa@univaq.it}

\author{Gianluca Crippa}
\address[G. Crippa]{Departement Mathematik und Informatik, Universit\"at Basel, Spiegelgasse 1, CH-4051 Basel, Switzerland.}
\email{gianluca.crippa@unibas.ch}

\date{\today}

\begin{document}

\vskip .2truecm

\begin{abstract}
    The goal of this paper is to study weak solutions of the Fokker-Planck equation. We first discuss existence and uniqueness of weak solutions in an irregular context, providing a unified treatment of the available literature along with some extensions. Then, we prove a regularity result for distributional solutions under suitable integrability assumptions, relying on a new, simple commutator estimate in the spirit of DiPerna-Lions' theory of renormalized solutions for the transport equation.
    Our result is somehow transverse to Theorem 4.3 of \cite{F08}: on the diffusion matrix we relax the assumption of Lipschitz regularity in time at the price of assuming Sobolev regularity in space, and we prove the regularity (and hence the uniqueness) of distributional solutions to the Fokker-Planck equation.
    
	\footnotesize{
		}
\end{abstract}

\maketitle
\tableofcontents

\section{Introduction}
This paper deals with regularity and uniqueness issues for weak solutions to the Fokker-Planck equation
\begin{equation}\label{eq:fp}\tag{FP}
\begin{cases}
\partial_t u + \dive(\b u) -\frac{1}{2} \sum_{i,j}\partial_{ij}(a_{ij}u)=0 &  \text{ in } (0,T) \times \R^d, \\
u\vert_{t=0}=u_0 & \text{ in } \R^d,
\end{cases}
\end{equation}
where $u:(0,T)\times\R^d\to\R$ is the unknown, $\b:(0,T)\times\R^d\to\R^d$ is a given vector field, $\a=(a_{ij})_{i,j=1,...,d}:(0,T)\times\R^d\to\R^{d\times d}$ is a given diffusion matrix and $u_0:\R^d\to\R$ is the initial datum. We are interested here in positive definite fixed diffusivity and not in behaviors for small (or vanishing) diffusivities and we refer to \cite{BCC, BN, CR2, DEIJ, LBL_CPDE, LiLuo, S21} for some recent results with degenerate viscosity coefficient. The equation \eqref{eq:fp} arises frequently in fluid-dynamics models where $u$ is a passive scalar which is simultaneously advected (by the given velocity field $\b$) and diffused by the second order operator associated with the matrix $\a$, see \cite{O, OP}. 
The Fokker-Planck equation is also intimately connected to the stochastic differential equation
\begin{equation}\label{eq:sde}\tag{SDE}
\begin{cases}
    \de X_{t,s}(x)=\b(s,X_{t,s}(x))\,\de s+\bsigma(s,X_{t,s}(x))\,\de W_s,\\
    X_{t,t}(x)=x,
\end{cases}
\end{equation}
where $\bsigma:(0,T)\times\R^d\to\R^{d\times d}$ is a matrix function and $W_s$ is a $d$-dimensional Brownian motion on a probability space $(\Omega,\mathcal{F},\mathbb{P})$. For sufficiently smooth $\b$ and $\bsigma$, by defining the matrix
\begin{equation}\label{def:sigma}
    \a(t,x)=\bsigma(t,x)\bsigma^T(t,x),
\end{equation} one has that the law of $X$ satisfies the equation \eqref{eq:fp} and the Feynman-Kac representation formula holds
\begin{equation}
    u(t,x):=\mathbb{E}[u_0(X_{t,0}(x))],
\end{equation}
where $\mathbb{E}$ denotes the expected value on the probability space $(\Omega,\mathcal{F},\mathbb{P})$. Roughly speaking, the connection between \eqref{eq:fp} and \eqref{eq:sde} is reminiscent of the one between the linear transport equation and the flow of the vector field $\b$ (which can be recovered as the end-point $\bsigma=0$) -- see however \cite{F08, LBL_CPDE, LBL_book} and references therein for a more in-depth discussion of these connections. 
Fokker-Planck equations are also relevant in Statistical Mechanics \cite{Frank-book, SNC}, Mean-Field Games theory \cite{CJPT, LaL, Por}, and Stochastic Analysis \cite{BR, CJ, F08, Ky, SV}. 

In this paper we are interested in the case in which both the driving vector field and the diffusion matrix are irregular, i.e. we will require only integrability and/or weak differentiability properties on $\b$ and $\a$. 
Besides the purely theoretical interest, this is a common situation in fluid dynamics models: we refer the reader to \cite[Section 4]{LBL_CPDE} for an overview of modeling polymeric fluids with SDEs and Fokker-Planck equations with irregular coefficients.

The well-posedness of the Cauchy problem in the smooth context is a classical result, see \cite{LA2}. Out of the smooth setting the main issue are \emph{uniqueness} and \emph{regularity} of solutions.
Indeed, existence results can be obtained by a simple approximation argument: under global $L^p$-bounds on the vector field and the diffusion matrix, one easily establishes energy estimates which allow to apply standard (weak) compactness results. The linearity of the equation ensures that the weak limit is a solution to \eqref{eq:fp}.
At a closer look, however, different a priori estimates are available for \eqref{eq:fp} and this is reflected in the presence of various notions of solutions. Understanding the relationships among different notions of solutions is one of the aims of the present work. We recall that this was the subject of study in a recent survey for constant diffusion matrix, see \cite{BCC2}.

\subsection{Distributional and parabolic solutions} Let $\b\in L^1_tL^p_x$, $\a\in L^1_tL^p_x$, and let $u_0 \in L^q_x$ for some $1 \le p, q\le \infty$ such that $\sfrac{1}{p}+\sfrac{1}{q} \le 1$. With these assumptions it is possible to introduce distributional solutions to \eqref{eq:fp}, i.e. functions $u \in L^\infty_t L^q_x$ solving the equation in the sense of distributions. Define the vector field
\begin{equation}\label{def: tilde b intro}
\tilde\b:=\b-\frac12\sum_j\partial_j a_{ij}.
\end{equation}
Note that, formally, the equation \eqref{eq:fp} is equivalent to
\begin{equation}\tag{FP-div}
    \label{eq:fp-div-intro}
    \partial_t u + \dive(\tilde\b u) -\frac{1}{2} \sum_{i,j}\partial_{i}(a_{ij}\partial_j u)=0,
\end{equation} 
which is also known as the Fokker-Planck equation in {\em divergence form}.
Then, we consider the following set of assumptions:\\
\\
\begin{pschema}
{\bf Assumptions }
\begin{enumerate}[label={\rm (A\arabic*)}]
\item \label{assu:bounded} the diffusion matrix $\a$ is uniformly bounded, i.e. 
\begin{equation*}
\a \in L^\infty((0,T);L^\infty(\R^d;\R^{d\times d})); 
\end{equation*}
\item \label{assu:dive_a_and_b} the diffusion matrix $\a$ satisfies
\begin{equation*}
	\sum_{j=1}^d \partial_j a_{ij}\in L^\infty((0,T)\times \R^d),\hspace{0.5cm} \mbox{for all }i\in  \{1,...,d\};
	\end{equation*}
\item \label{ass:dive b} the vector field $\tilde\b$ satisfies
\begin{equation*}\label{ass:2}
	(\dive \tilde{\b})^{-}=\bigg(\dive \b -\frac{1}{2}\sum_{i,j=1}^d\partial_{ij} a_{ij}\bigg)^-\in L^\infty((0,T)\times\R^d);
\end{equation*}
\item \label{assu:elliptic} the diffusion matrix is uniformly elliptic, i.e. there exists $\alpha>0$ such that for every $\xi \in \R^d$ it holds  
\begin{equation*}
\langle \xi, \a(t,x)\xi\rangle\geq \alpha|\xi|^2, \hspace{0.3cm} \text{for a.e. }(t,x)\in (0,T)\times\R^d. 
\end{equation*}
\end{enumerate}
\end{pschema}
\vspace{0.4cm}

One a priori estimate for smooth solutions is the following
\begin{equation}
\|u(t,\cdot)\|_{L^q}^q\leq \|u_0\|_{L^q}^q+(q-1)\|(\dive\tilde{\b})^{-}\|_\infty\int_0^t\|u(s,\cdot)\|_{L^q}^q\de s,
\end{equation}
and a standard approximation argument provides the existence of a distributional solution to \eqref{eq:fp}. 
However, distributional solutions are not, in general,  unique, not even in the case of constant coefficient diffusion matrices, see \cite{MS3, MS, MS2}.
This motivates the introduction of another notion of solution. Exploiting the presence of the second order operator in the equation, one can show another energy estimate for smooth solutions, namely
\begin{equation*}
\|u(t,\cdot)\|_{L^2}^2+\alpha\int_0^t\int_{\R^d} |\nabla u(s,x)|^2\de x\de s\leq \|u_0\|_{L^2}^2-\int_0^t\int_{\R^d}\dive\tilde{\b}(s,x)|u(s,x)|^2\de x\de s
\end{equation*} 
for every $t \in [0,T]$, which suggests that one should look for solutions possessing $L^2$ gradient, i.e. solutions that are $H^1$ in the space variable. We therefore say that a distributional solution $u \in L^\infty_t L^q_x$ to \eqref{eq:fp} is \emph{parabolic} if it holds $u \in L^2_t H^1_x$. 

Parabolic solutions carry the exact regularity needed to establish their uniqueness under suitable integrability assumptions on $\b$ and $\a$. In the setting of rough drift/diffusion coefficients, two uniqueness results for parabolic solutions are available under slightly different assumptions:
\begin{itemize}
    \item in \cite{F08} it is shown that a parabolic solution $u\in L^\infty_tL^2_x\cap L^2_tH^1_x$ is unique provided that $u_0\in L^2_x$ and $\b \in L^\infty_{t,x}$ satisfy \ref{assu:bounded}, \ref{assu:dive_a_and_b}, \ref{ass:dive b} and \ref{assu:elliptic};
    \item in \cite{LBL_CPDE} it is proved that a parabolic solution $u\in L^\infty_t(L^2_x\cap L^\infty_x)\cap L^2_tH^1_x$ is unique in the set
    $$
    \{ u\in L^2_tH^1_x: \bsigma^T\nabla u\in L^2_{t,x}\},
    $$
    where $\bsigma$ is defined accordingly to \eqref{def:sigma}, provided that
    \begin{itemize}
    \item $u_0\in L^2_x\cap L^\infty_x$; 
        \item $\b\in L^2_tL^2_{x,\mathrm{loc}}$ satisfy the growth assumption
        $
        \frac{\b}{1+|x|}\in L^1_tL^1_x+L^1_tL^\infty_x;
        $
        \item the matrix $\bsigma\bsigma^T$ is uniformly positive definite;
        \item $\bsigma\in L^\infty_tL^\infty_{x,\mathrm{loc}}$ satisfy the growth assumption
        $
        \frac{\bsigma}{1+|x|}\in L^2_tL^2_x+L^2_tL^\infty_x;
        $
        \item $\dive\tilde\b \in L^1_tL^\infty_x$.
    \end{itemize}
\end{itemize}
The main difference between the two results lies in the $L^\infty$ assumption: \cite{F08} assumes $\b,\a\in L^\infty_{t,x}$, while \cite{LBL_CPDE} assumes $u_0\in L^\infty_x$.
Moreover, the proofs of these two results are based on different techniques: in \cite[Theorem 4.3]{F08} the author uses a functional analytic approach closer to the one commonly used for elliptic/parabolic equations which make use of a Lax-Milgram type argument; on the other hand, the result in \cite{LBL_CPDE} is more ``hyperbolic" in nature resorting to commutator estimates and on the theory of renormalized solutions as done in \cite{DPL} for the linear transport equation. 
In this regards, the $L^2_tH^1_x$ regularity of the solution allows to obtain a better control on the so called \emph{commutator}, which measures the error one commits when considering smooth approximations of the solution. In particular, in \cite{LBL_CPDE, LBL} it is shown that the commutator for parabolic solutions converges strongly to $0$ in $L^1_{t,x}$, but better bounds can be established for parabolic solutions $u \in L_t^2H_x^1$ if $\b \in L^2_t L^2_x$ (see Section 2 for a comprehensive list of old and new commutator estimates).

Relying on these techniques, in Section 3 we provide a presentation of the available results on the well-posedness of \eqref{eq:fp} together with some technical extensions.

We make use of the new commutator estimates to prove the uniqueness of parabolic solutions of \eqref{eq:fp} in a setting that incorporates and generalises those of \cite{F08, LBL_CPDE}, thus providing a unified treatment of the existing literature. Our theorem is the following.
\begin{thm}\label{thm:unicitanostro}
Let $u_0\in L^2\cap L^q(\R^d)$ be a given initial datum and assume that
\begin{itemize}
    \item[(i)] $\b\in L^2((0,T); L^p(\R^d;\R^d))$ with $1/p+1/q= 1/2$;
    \item[(ii)] $\b$ satisfies the growth condition $$\frac{\b}{1+|x|}\in L^1((0,T);L^1(\R^d;\R^d))+L^1((0,T);L^\infty(\R^d;\R^d));$$
    \item[(iii)] $\a$ satisfies \ref{assu:bounded}, \ref{assu:dive_a_and_b} and \ref{assu:elliptic};
    \item[(iv)] the vector field $\tilde \b$ defined in \eqref{def: tilde b intro} satisfies
    $$
    (\dive\tilde \b)^-\in L^1((0,T);L^\infty(\R^d)).
    $$
\end{itemize}
Then, there exists a unique parabolic solution $u$ to \eqref{eq:fp}.
\end{thm}

Note that if $p=\infty$ and $q=2$ we precisely recover the existence and uniqueness of parabolic solutions to \eqref{eq:fp} in \cite[Theorem 4.3]{F08}, while when $p=2$ and $q=\infty$ we obtain \emph{essentially}  the corresponding result in \cite{LBL_CPDE} -- see however Subsection \ref{sss:comparison} below for a detailed comparison between the two results. 

\subsection{Regularity and uniqueness of distributional solutions}
According to our definitions, parabolic solutions cannot always be defined, but if they can, then they are always distributional. The converse implication is, in general, not true: in \cite{MS3} it is shown that there exist infinitely many distributional solutions $u\in L^\infty_t L^2_x$ to \eqref{eq:fp} with a divergence-free vector field $\b \in L^\infty_t L^2_x$ and $\a=\mathrm{Id}$, while the parabolic one is unique. This motivates the search for a condition that guarantees \emph{parabolic regularity} of a distributional solution. In addition to the uniqueness of parabolic solutions, in \cite{F08} it is proved that the condition 
\[
\partial_t a_{ij}\in L^\infty((0,T)\times\R^d)
\] 
for $i,j=1,...,d$, implies 
that every solution $u\in L^2((0,T)\times\R^d)$ to \eqref{eq:fp} belongs to the space 
$$
Y:=\{u\in L^2((0,T);H^1(\R^d)):\partial_t u\in L^2((0,T); H^{-1}(\R^d))\}.
$$

The idea of the proof is simple: once the solutions in $Y$ are proved to be unique, uniqueness in the functional space $L^2_{t,x}$ follows from a regularity argument which makes use of the condition $\partial_t \a\in L^\infty_{t,x}$. Our purpose is to investigate the role of the assumption on the time derivative of $\a$. In this regard, we show that if we replace $\partial_t a_{ij}\in L^\infty_{t,x}$ with some Sobolev regularity in space, distributional solutions are actually more regular. Our theorem reads as follows: 
\begin{thm}\label{thm:main1}
 Let $p,q \ge 1$ such that $\sfrac{1}{p}+\sfrac{1}{q}\leq \sfrac{1}{2}$. Assume that $\b\in L^2((0,T); L^p(\R^d;\R^d))$ and $\a\in L^2((0,T);W^{1,p}(\R^d;\R^{d \times d}))$ satisfy \ref{assu:dive_a_and_b}, \ref{assu:elliptic}, and $(\dive \tilde \b)^- \in L^1((0,T);L^\infty(\R^d))$. Let $u\in L^\infty((0,T);L^q(\R^d))$ be a distributional solution to \eqref{eq:fp}, then $u\in L^2((0,T);H^1_\mathrm{loc}(\R^d))$.
\end{thm} 

The proof relies again on a new commutator estimate: we show that, in the current regime, the convergence to zero of the commutators takes place in $L^2_tH^{-1}_{x,\mathrm{loc}}$ and this, together with the energy estimate, implies the regularity of distributional solutions. Moreover, in the particular case $p=\infty, q=2$ we obtain the uniqueness of distributional solutions in $L^2_{t,x}$ -- notice that here we need the stronger assumption \ref{ass:dive b} instead of the weaker $L^1_tL^\infty_x$ control on $(\dive \tilde{\b})^-$. 
\begin{thm}\label{cor:fig}
Assume that $\b\in L^\infty((0,T);L^\infty(\R^d))$ and that \ref{assu:bounded}, \ref{assu:dive_a_and_b}, \ref{ass:dive b} and \ref{assu:elliptic} are fulfilled. Moreover, assume that $\nabla\a\in L^\infty((0,T)\times \R^d)$. Then, any distributional solution $u\in L^2((0,T)\times\R^d)$ of \eqref{eq:fp} belongs to $Y$ and thus it is unique.
\end{thm}

The theorem above is somehow transverse to \cite[Theorem 4.8]{F08}, i.e. we replace the assumption on the time derivative $\partial_t \a \in L^\infty_{t,x}$ with an assumption on the spatial gradient $\nabla \a \in L^\infty_{t,x}$. Note that we need the $L^\infty_t$ assumption on $\nabla\a$ to compensate the $L^2_t$ assumption on $u$.

Furthermore, the \emph{local} regularity in space  $H^1_\mathrm{loc}$ provided by Theorem \ref{thm:main1} suggests investigating whether uniqueness of solutions to \eqref{eq:fp} holds in this larger functional space.
We will say that a distributional solution $u \in L^\infty_t L^q_x$ to \eqref{eq:fp} is \emph{locally parabolic} if it holds $u \in L^2_t H^1_{x,\mathrm{loc}}$. Assuming suitable growth conditions on $\b,\a$ (similarly to \cite{DPL, LBL_CPDE}), we prove that distributional solutions are locally parabolic and unique in the class $L^\infty((0,T);L^q(\R^d))$:

\begin{thm}\label{thm:main2}
Let $u_0\in L^q\cap L^\infty(\R^d)$ be a given initial datum for some $q>2$ and assume that $\b\in L^2((0,T);L^p(\R^d))$ and $\a\in L^2((0,T);W^{1,p}(\R^d))$ with $1/p+1/q=1/2$ satisfy \ref{assu:dive_a_and_b} and \ref{assu:elliptic}. Moreover, we assume the following growth conditions
\begin{equation}
\frac{\b(t,x)}{1+|x|},\frac{\partial_j a_{ij}(t,x)}{1+|x|},\frac{\a(t,x)}{1+|x|^2}\in L^1((0,T);L^1(\R^d))+L^1((0,T);L^\infty(\R^d)).
\end{equation}
Then, there exists at most one distributional solution $u\in L^\infty((0,T);L^q(\R^d))\cap L^2((0,T);H^1_\mathrm{loc}(\R^d))$.
\end{thm}
In the proof we use the assumptions on $p,q$ and the Regularity Theorem \ref{thm:main1} to show the existence of a solution in $L^\infty((0,T);L^q(\R^d))\cap L^2((0,T);H^1_\mathrm{loc}(\R^d))$. Then we exploit the growth conditions on $\b$ and $\a$ to use a Gronwall type argument following the same strategy of Theorem \ref{thm:unicitanostro}. Finally, we show how we can play with the integrability exponents in the growth conditions to guarantee that a distributional solution is in $L^2((0,T); H^1(\R^d))$, obtaining the following result: 
\begin{thm}\label{teo:growth}
Let $u\in L^\infty((0,T); L^2\cap L^q(\R^d))$ be a distributional solution of \eqref{eq:fp}, and let $\b\in L^2((0,T);L^p(\R^d))$ and $\a\in L^2((0,T);W^{1,p}(\R^d))$ with $1/p+1/q = 1/2$ satisfy \ref{assu:dive_a_and_b} and \ref{assu:elliptic}. Moreover, assume that $2<q\leq \frac{2d}{d-2}$ together with the following growth conditions
$$
\frac{\b(t,x)}{1+|x|},\frac{\partial_j a_{ij}(t,x)}{1+|x|}, \frac{\a(t,x)}{1+|x|^2}\in L^1((0,T);L^{\frac{q}{q-2}}(\R^d))+L^1((0,T);L^\infty(\R^d)).
$$
Then $u\in L^2((0,T);H^1(\R^d))$. 
\end{thm}

We conclude this introduction by observing that the well-posedness landscape of \eqref{eq:fp} in the non-smooth regime is far from being completely understood. Recently, several authors have devoted their attention to the construction of \emph{exotic} non-unique distributional solutions to Fokker-Planck type equations (and also for related PDEs in fluid-dynamics). These counterexamples are mainly provided in the case of constant diffusion matrix -- we refer the reader to our recent survey \cite{BCC2} for a thorough analysis of the problem and for additional references.  
From the mathematical point of view, handling the PDE in the presence of irregular diffusions is even more delicate. At the same time, anisotropic diffusions naturally appear in the physical models and therefore the study of \eqref{eq:fp} in this setting is of paramount importance.


\subsection{Structure of the paper} In Section 2 we collect the commutator estimates which will then be the heart of all the proofs that follow. In Section 3 we recall the definitions of distributional and parabolic solutions of \eqref{eq:fp} and we prove their existence together with the uniqueness of parabolic solutions. These results are technical improvements of \cite{F08, LBL_CPDE}, providing a unified treatment. Finally, in Section 4 we prove our Regularity Theorem \ref{thm:main1} and Theorem \ref{cor:fig}. Moreover, we discuss the role of the growth conditions, proving Theorem \ref{thm:main2} and Theorem \ref{teo:growth}.

In what follows, we will often omit the summation symbol and we will adopt the convention of summing over repeated indices.

\section{Commutator estimates}
In this section we collect commutator estimates that will play a predominant role in the entire paper. First, we recall some known results on the $L^1$ convergence of the commutator. We provide the proofs for completeness but they can be found in \cite{DPL, LBL_CPDE, LBL_book}. We remark that the convergence of the commutators is the key tool to prove the renormalization property (and hence the uniqueness) of distributional solutions of the linear transport equation, see \cite{DPL}. To fix the notation, let $\rho\in C^\infty_c(\R^d)$ be a smooth convolution kernel such that
\begin{equation*}
    \int_{\R^d}\rho(x)\de x=1,
\end{equation*}
and for any $\delta>0$ define $\rho^\delta$ as 
\begin{equation}\label{def:moll}
    \rho^\delta(x):=\frac{1}{\delta^d}\rho\left(\frac{x}{\delta}\right).
\end{equation}
\subsection{Commutator estimates in $L^1$} We start by proving some classical commutator estimates. The following lemma holds.
\begin{lem}\label{lem:conv_comm}
Let $\b \in L^2((0,T); L^p_\mathrm{loc}(\R^d;\R^d))$ be a vector field and let $w \in L^\infty((0,T); L^q_\mathrm{loc}(\R^d))$, where $p,q$ are positive real numbers with $\sfrac{1}{p}+\sfrac{1}{q}\leq 1$. Let $(\rho^\delta)_\delta$ be a family of smooth convolution kernels. Define the \emph{commutator of $w$ and $\b$} as follows:
\begin{equation}\label{eq:comm_def}
r^\delta:=\dive(\b w^\delta)-\dive(\b w)*\rho^\delta=r^\delta_1+r^\delta_2,
\end{equation}
where
\begin{equation}\label{eq:comm_def1}
r^\delta_1:=\b\cdot\nabla (w*\rho^\delta)-\left(\b\cdot\nabla w\right)*\rho^\delta,\qquad
r^\delta_2:=w*\rho^\delta\dive\b-(w\dive\b)*\rho^\delta.
\end{equation}
Then we have the following:
\begin{itemize}
\item if $\nabla w\in L^2((0,T);L^q_\mathrm{loc}(\R^d;\R^d))$, then $r_1^\delta$ converges to $0$ in $L^1((0,T);L^1_\mathrm{loc}(\R^d))$;
\item if $\dive \b \in L^1((0,T);L^\infty(\R^d))$, then $r_2^\delta$ converges to $0$ in $L^1((0,T); L^q_\mathrm{loc}(\R^d))$.
\end{itemize}
\end{lem}
\begin{proof}
We start by showing the convergence of $r^\delta_1$. Observe that, for a.e. $t \in (0,T)$ and a.e. $x \in \R^d$, we can explicitly write the commutator $r^\delta_1$ in the following form:
	\begin{align*}
	r^\delta_1(t,x)& =[\b\cdot\nabla (w*\rho^\delta)](t,x)-[(\b\cdot\nabla w)*\rho^\delta](t,x)\\
	&=\b(t,x)\cdot\nabla \int_{\R^d}w(t,x-y)\rho^\delta(y)\de y-\int_{\R^d} \b(t,x-y)\cdot \nabla w(t,x-y) \rho^\delta(y)\de y \\
	&=\int_{\R^d}\rho^\delta(y) \left(\b(t,x)-\b(t,x-y)\right)\cdot\nabla w(t,x-y) \de y\\
	&=\int_{\R^d}\rho(z)\left(\b(t,x)-\b(t,x-\delta z)\right)\cdot\nabla w(t,x-\delta z) \de z.
	\end{align*}
Let $Q$ be any bounded subset of $\R^d$, we have that
	\begin{align}
	\iint_{(0,T) \times Q}|r^\delta_1(t,x)|\de t \de x  &=\iint_{(0,T) \times Q}\left| \int_{\R^d}\rho(z)\left(\b(t,x)-\b(t,x-\delta z)\right)\cdot\nabla w(t,x-\delta z) \de z\right|\de t \de x  \nonumber\\
	&\leq \int_{B_1}  \rho(z)\int_0^T \int_{Q}|\b(t,x)-\b(t,x-\delta z)||\nabla w(t,x-\delta z)|\de x \de t \de z.\label{eq:a}
	\end{align}
Since $(t,x) \mapsto \b(t,x)-\b(t,x-\delta z)$ converges to $0$ in measure (for every fixed $z$), the conclusion follows by the Dominated Convergence Theorem. \\
Now we consider the commutator $r_2^\delta$: observe that $w\ast \rho^\delta \to w$ strongly in $L^\infty L^q_\mathrm{loc}$ and similarly $(w\dive \b) \ast \rho^\delta \to w \dive \b$ strongly in $L^\infty L^q_\mathrm{loc}$ and this concludes the proof.
\end{proof}

\begin{rem}\label{rem:div-commutatore1}
It is worth observing that no assumption on the divergence of $\b$ is necessary for the convergence of $r_1^\delta$.
\end{rem}

\begin{lem}\label{lem:conv_comm3}
	Let $w \in L^2((0,T); W^{1,q}_\mathrm{loc}(\R^d))$ and let $\a\in L^2((0,T); W^{1,p}_\mathrm{loc}(\R^d;\R^{d\times d}))$ be a matrix-valued function, where $p,q$ are positive real numbers with $\sfrac{1}{p}+\sfrac{1}{q}\leq 1$. Let $(\rho^\delta)_\delta$ be a family of smooth convolution kernels defined as in \eqref{def:moll}. Define the \emph{commutator of $w$ and $\a$} as follows: 
    \begin{equation}\label{eq:comm2_def}
	s^\delta:=\sum_{i,j}[\partial_{ij}(a_{ij}w)]*\rho_\delta-\partial_{ij}[a_{ij}(w*\rho_\delta)].
    \end{equation}
    Then $s^\delta$ converges to $0$ in $L^1((0,T);L^1_\mathrm{loc}(\R^d))$. 
\end{lem}
\begin{proof}
We can rewrite the commutator $s^\delta$ as follows
\begin{align*}
s^\delta&=\sum_{i,j}[\partial_{ij}(a_{ij}w)]*\rho_\delta-\partial_{ij}[a_{ij}(w*\rho_\delta)]\\
&=\sum_{i,j}[\partial_{i}(a_{ij}\partial_jw)]*\rho_\delta-\partial_{i}[a_{ij}\partial_j(w*\rho_\delta)]+\sum_{i,j}[\partial_{i}(\partial_ja_{ij}w)]*\rho_\delta-\partial_{i}[\partial_ja_{ij}(w*\rho_\delta)]\\
&:=s_1^\delta+s_2^\delta.
\end{align*}
We start by analyzing $s_1^\delta$: for a.e. $(t,x)\in (0,T)\times \R^d$ we can write
\begin{align*}
s_1^\delta(t,x)&=\sum_{i,j}[\partial_{i}(a_{ij}\partial_j w)]*\rho_\delta(t,x)-\partial_{i}[a_{ij}\partial_j(w*\rho_\delta)](t,x)\\
&=\sum_{i,j}\int_{\R^d} \partial_{i}[a_{ij}(t,x-y)\partial_jw(t,x-y)]\rho_\delta(y)\de y-\partial_{i} \left(a_{ij}(t,x)\partial_j\int_{\R^d}w(t,x-y)\rho_\delta(y)\de y\right)\\
&=\sum_{i,j}\int_{\R^d} \partial_{i}[(a_{ij}(t,x-y)-a_{ij}(t,x))\partial_j w(t,x-y)]\rho_\delta(y)\de y\\
&=-\sum_{i,j}\int_{\R^d}\frac{a_{ij}(t,x-\delta z)-a_{ij}(t,x)}{\delta}\partial_j w(t,x-\delta z)\partial_i\rho(z)\de z\\
&=\sum_{i,j}\int_{\R^d}\int_0^1\partial_j w(t,x-\delta z)\nabla a_{ij}(t,x-s\delta z)\cdot z\partial_i\rho(z)\de s\de z,
\end{align*}
where in the fourth line we changed variables and integrated by parts. Then, for a.e. $(t,x)$ the commutator $s_1^\delta$ converges to
\begin{align}
\lim_{\delta\to 0}s_1^\delta(t,x)&=\sum_{i,j}\partial_jw(t,x)\nabla a_{ij}(t,x)\cdot\int_{\R^d} z\partial_i\rho(z)\de z\nonumber\\
&=-\sum_{i,j}\partial_jw(t,x)\partial_i a_{ij}(t,x)\int_{\R^d} \rho(z)\de z.\label{comm:s1}
\end{align}
We now focus on $s_2^\delta$: we have that
\begin{align}
s_2^\delta(t,x)&=\sum_{i,j}[\partial_{i}(\partial_ja_{ij}w)]*\rho_\delta(t,x)-\partial_{i}[\partial_ja_{ij}(w*\rho_\delta)](t,x)\nonumber\\
=&\sum_{i,j}\int_{\R^d} \partial_{i}[\partial_ja_{ij}(t,x-y)w(t,x-y)]\rho_\delta(y)\de y-\partial_{i} \left(\partial_ja_{ij}(t,x)\int_{\R^d}w(t,x-y)\rho_\delta(y)\de y\right)\nonumber\\
=&\sum_{i,j}\int_{\R^d} \partial_{i}[\partial_j(a_{ij}(t,x-y)-a_{ij}(t,x)) w(t,x-y)]\rho_\delta(y)\de y\nonumber\\
=&-\sum_{i,j}\int_{\R^d}\frac{\partial_ja_{ij}(t,x-\delta z)-\partial_ja_{ij}(t,x)}{\delta} w(t,x-\delta z)\partial_i\rho(z)\de z\nonumber\\
=&\sum_{i,j}\int_{\R^d}\frac{a_{ij}(t,x-\delta z)-a_{ij}(t,x)}{\delta} \partial_jw(t,x-\delta z)\partial_i\rho(z)\de z\nonumber\\
&+\sum_{i,j}\int_{\R^d}\frac{a_{ij}(t,x-\delta z)-a_{ij}(t,x)}{\delta} w(t,x-\delta z)\partial_{ij}\rho(z)\de z\nonumber\\
=& \underbrace{-\sum_{i,j}\int_{\R^d}\int_0^1\partial_j w(t,x-\delta z)\nabla a_{ij}(t,x-s\delta z)\cdot z\partial_i\rho(z)\de s\de z}_{I^\delta}\nonumber\\
&\underbrace{-\sum_{i,j}\int_{\R^d}\int_0^1 w(t,x-\delta z)\nabla a_{ij}(t,x-s\delta z)\cdot z\partial_{ij}\rho(z)\de s\de z}_{II^\delta}.\nonumber
\end{align}
Thus, one can check the following a.e. convergences
\begin{equation}
\lim_{\delta\to 0}I^\delta(t,x)=\sum_{i,j}\partial_jw(t,x)\partial_i a_{ij}(t,x)\int_{\R^d} \rho(z)\de z,\label{comm:s2}
\end{equation}
and
\begin{equation}
\lim_{\delta\to 0}II^\delta(t,x)=-w(t,x)\sum_{i,j} \nabla a_{ij}(t,x)\cdot \int_{\R^d}z\partial_{ij}\rho(z)\de s\de z=0.\label{comm:s3}
\end{equation}
We have thus shown that $s^\delta\to 0$ for a.e. $(t,x)\in(0,T)\times\R^d$. The convergence in $L^1((0,T);L^1_\mathrm{loc}(\R^d))$ follows from the dominated convergence theorem.
\end{proof}

\subsection{Commutator estimates on negative Sobolev norms}
We now show some new commutator estimates which will allow us to prove the regularity of a distributional solution within an appropriate integrability range.
\begin{lem}\label{lem:conv_comm2}
Let $\b \in L^2((0,T); L^p_\mathrm{loc}(\R^d; \R^d))$ be a vector field and let $w \in L^\infty((0,T); L^q_\mathrm{loc}(\R^d))$, where $p,q$ are positive real numbers with $\sfrac{1}{p}+\sfrac{1}{q}\leq \sfrac{1}{2}$. Let $(\rho^\delta)_\delta$ be a family of smooth convolutions kernels and define $r^\delta$ as in \eqref{eq:comm_def}. Then $r^\delta$ converges to $0$ in $L^2((0,T); H^{-1}_\mathrm{loc}(\R^d))$.
\end{lem}
\begin{proof}
	We recall that the commutator $r^\delta$ is given by
	\begin{equation*}
		\begin{split}
			r^\delta &  =\dive [\b (w*\rho^\delta)]-[\dive (\b w)*\rho^\delta] = \dive [\b (w*\rho^\delta) - (\b w)*\rho^\delta], 
		\end{split}
	\end{equation*}
	in the sense of distributions on $(0,T) \times \R^d$. We can thus write 
	\begin{equation*}
		r^\delta(t,x) =  \dive_x \left( \int_{\R^d} [\b(t,x) - \b(t,x-y)] w(t,x-y)\rho^\delta(y)\de y  \right)
	\end{equation*}
	and, for a given bounded subset $Q$ of $\R^d$ we can estimate  
	\begin{align*}
		\|r^\delta &\|_{L^2(H^{-1}(Q))}=\sup_{\|\varphi\|_{L^2 H^1_0(Q)}\leq 1}\left|\iint_{(0,T) \times Q}r^\delta(t,x)\varphi(t,x)\de t\de x\right|\\ 
			& = \sup_{\|\varphi\|_{L^2 H^1_0(Q)}\leq 1} \left \vert \iint_{(0,T) \times Q} \left( \int_{\R^d} [\b(t,x) - \b(t,x-y)] w(t,x-y)\rho^\delta(y)\de y  \right)\cdot \nabla \varphi(t,x) \de t\de x \right \vert \\
			& \le \sup_{\|\varphi\|_{L^2 H^1_0(Q)}\leq 1} \int_{B_1} \rho(z) \int_0^T \int_{Q}|\b(t,x)-\b(t,x-\delta z)||w(t,x-\delta z)| |\nabla\varphi(t,x)|\de x \de t \de z.  
	\end{align*}
	Notice now that, as in the proof of Lemma \ref{lem:conv_comm}, the map $(t,x) \mapsto \b(t,x)-\b(t,x-\delta z)$ converges to $0$ in measure (for every fixed $z$). H\"older inequality on the product space $(0,T) \times Q$ with exponents $(p,q,2)$ (in space) and $(2,\infty, 2)$ (in time) allows us to apply Lebesgue Dominated Convergence Theorem and we can therefore conclude that $r^\delta \to 0$ in $L^2(H^{-1}_\mathrm{loc})$. 
\end{proof}

\begin{rem}\label{rem:global_comm_h-1}
    Notice that with global assumptions on $\b$ and $w$, in the case $1/p+1/q=1/2$ the commutator $r^\delta$ converges to $0$ in $L^2((0,T);H^{-1}(\R^d))$.
\end{rem}

We now prove the analogous result on the diffusion matrix.

\begin{lem}\label{lem:conv_comm4}
Let $\a\in L^2((0,T); W^{1,p}_\mathrm{loc}(\R^d;\R^{d\times d}))$ be a matrix valued function  and let $w \in L^\infty((0,T); L^q_\mathrm{loc}(\R^d))$ with $p,q$ are positive real numbers with $\sfrac{1}{p}+\sfrac{1}{q}\leq \sfrac{1}{2}$. Let $(\rho^\delta)_\delta$ be a family of smooth convolutions kernels and define $s^\delta$ as in \eqref{eq:comm2_def}. Then $s^\delta$ converge to $0$ in $L^2((0,T); H^{-1}_\mathrm{loc}(\R^d))$.
\end{lem}
\begin{proof}
Let $Q$ be any bounded subset of $\R^d$. Then, for any $\varphi\in L^2((0,T);H^{1}_0(Q))$ we compute
\begin{align*}
	&\iint_{(0,T) \times Q}s^\delta(t,x)\varphi(t,x)\de t\de x=\iint_{(0,T) \times Q}\bigg( \sum_{i,j}[\partial_{ij}(a_{ij}w)]*\rho_\delta-\partial_{ij}[a_{ij}(w*\rho_\delta)] \bigg) \varphi(t,x)\de t\de x\\
	= & \sum_{i,j}\iint_{(0,T) \times Q}\partial_{ij}\left(\int_{\R^d}a_{ij}(t,y)w(t,y)\rho_\delta(x-y)\de y-a_{ij}(t,x)\int_{\R^d}w(t,y)\rho_\delta(x-y)\de y\right) \varphi(t,x)\de t\de x\\
	=&  -\sum_{i,j}\iint_{(0,T) \times Q}\partial_{j}\left(\int_{\R^d}a_{ij}(t,y)w(t,y)\rho_\delta(x-y)\de y-a_{ij}(t,x)\int_{\R^d}w(t,y)\rho_\delta(x-y)\de y\right) \partial_i\varphi(t,x)\de t\de x\\
   = & -\sum_{i,j}\int_0^T\int_{\R^d}\int_{Q} \partial_i \varphi(t,x)w(t,y)\frac{1}{\delta^{d+1}}\partial_{j}\rho\left(\frac{x-y}{\delta}\right)\left(a_{ij}(t,y)-a_{ij}(t,x)\right)\de x\de y\de t\\
	& +\sum_{i}\int_0^T\int_{Q}\int_{\R^d}w(t,y)\rho_\delta(x-y) \partial_i \varphi(t,x) \bigg(\sum_j\partial_{j} a_{ij}(t,x)\bigg)\de y \de t\de x,
\end{align*}
and by using a change of variables and Fubini to exchange integrals we finally get that
\begin{align*}
  &\iint_{(0,T) \times Q}s^\delta(t,x)\varphi(t,x)\de t\de x\\
  = &-\sum_{i}\int_0^T\int_{\R^d}\int_{Q}\partial_i \varphi(t,x)w(t,x-\delta z)\sum_j\left(\partial_{j}\rho(z)\left(\frac{a_{ij}(t,x-\delta z)-a_{ij}(t,x)}{\delta}\right)\right)\de x\de z\de t\\
	& +\sum_{i}\int_0^T\int_{Q}\int_{\R^d}w(t,x-\delta z)\rho(z) \partial_i \varphi(t,x) \bigg(\sum_j\partial_{j} a_{ij}(t,x)\bigg)\de z \de t\de x.
\end{align*}
Thus, by letting $\delta \to 0$ we obtain
\begin{equation*}
		\begin{split}
		\lim_{\delta\to 0} \iint_{(0,T) \times Q} s^\delta(t,x)\varphi(t,x)\de t\de x  = & \sum_{i,j}\int_0^T\int_{\R^d}\int_{Q}\partial_i \varphi(t,x)w(t,x)\partial_{j}\rho(z)\nabla a_{ij}(t,x)\cdot z\de x\de z\de t\\
        &+\sum_{ij}\int_0^T\int_{Q}\int_{\R^d}w(t,x)\rho(z) \partial_i \varphi(t,x) \partial_{j} a_{ij}(t,x)\de z \de t\de x\\
			= & -\sum_{i,j}\int_0^T\int_{\R^d}\int_{Q}\partial_i \varphi(t,x)w(t,x)\rho(z)\nabla a_{ij}(t,x)\cdot e_j\de x\de z\de t\\
			&+\sum_{ij}\int_0^T\int_{Q}\int_{\R^d}w(t,x)\rho(z) \partial_i \varphi(t,x) \partial_{j} a_{ij}(t,x)\de z \de t\de x =0,
		\end{split}	
\end{equation*}
where in the last step we have integrated by parts in the $z$-variable. This concludes the proof.
\end{proof}

We conclude this section with the following commutator estimate: it will be crucial to recover the uniqueness of parabolic solutions in the setting of \cite{F08}.
\begin{lem}
\label{lem:conv_comm5}
Let $\a \in L^\infty((0,T); L^\infty_\mathrm{loc}(\R^d;\R^{d\times d}))$ be a matrix valued function and let $\nabla w \in L^2((0,T);L^2_\mathrm{loc}(\R^d))$. Let $(\rho^\delta)_\delta$ be a family of smooth convolutions kernels and consider the commutator $s_1^\delta$ as follows 
\begin{equation}\label{def:s_1}
s_1^\delta:=\sum_i\partial_i\bigg(\sum_j a_{ij}\partial_j w^\delta-(a_{ij}\partial_j w)*\rho^\delta\bigg).
\end{equation}
Then $s_1^\delta$ converges to $0$ in $L^2((0,T); H^{-1}_\mathrm{loc}(\R^d))$.
\end{lem}
\begin{proof}
The proof is very similar to the one of Lemma \ref{lem:conv_comm2} and relies on the divergence form of the commutator.
Given a bounded subset $Q$ of $\R^d$, we can estimate  
	\begin{align*}
		\|s_1^\delta & \|_{L^2(H^{-1}(Q))}=\sup_{\|\varphi\|_{L^2 H^1_0(Q)}\leq 1}\left|\iint_{(0,T) \times Q}s_1^\delta(t,x)\varphi(t,x)\de t\de x\right|\\ 
		& = \sup_{\|\varphi\|_{L^2 H^1_0(Q)}\leq 1} \left \vert \iint_{(0,T) \times Q} \sum_i\left( \sum_j a_{ij}\partial_j w^\delta-(a_{ij}\partial_j w)*\rho^\delta\right) \partial_i \varphi(t,x) \de t\de x \right \vert \\
        &\le \sup_{\|\varphi\|_{L^2 H^1_0(Q)}\leq 1}\int_{B_1}\rho(z)\int_0^T\int_Q |\a(t,x)-\a(t,x-\delta z)||\nabla w(t,x-\delta z)||\nabla\varphi(t,x)|\de x \de t\de z.
	\end{align*}
	As in the proof of Lemma \ref{lem:conv_comm}, the map $(t,x) \mapsto \a(t,x)-\a(t,x-\delta z)$ converges to $0$ in measure (for every fixed $z$) and then by H\"older inequality on the product space $(0,T) \times Q$ with exponents $(\infty,2,2)$  we can apply Lebesgue Dominated Convergence Theorem obtaining that $s_1^\delta \to 0$ in $L^2((0,T);H^{-1}_\mathrm{loc}(\R^d))$. 
\end{proof}

\begin{rem}
It is worth to note that the commutator $s_1^\delta$ does not converge pointwise a.e. to $0$, as shown in \eqref{comm:s1}. However it converges in the sense of the distributions and, by density, in the functional space $L^2_tH^{-1}_\mathrm{loc}$.
\end{rem}

\section{Well-posedness of the Cauchy problem}
In this section we collect some results on the Cauchy problem for the Fokker-Planck equation \eqref{eq:fp}. Our aim is to give a complete presentation of the available results and to provide technical extensions. In this regard, we first analyse the Fokker-Planck equation in divergence form, providing some technical improvements of the results in \cite{LBL_CPDE}. Then we move to the analysis of the Fokker-Planck equation, proving existence and uniqueness in a slightly more general setting than the one considered in \cite{F08}. Theorem \ref{thm:uniqueness_weak_parabolic_FP} somehow interpolates the results from \cite{LBL_CPDE} and \cite{F08} using the $L^2H^{-1}_\mathrm{loc}$ convergence of the commutators.

\subsection{The equation in divergence form}
In this section we study the initial value problem
\begin{equation}\label{eq:divFP}\tag{FP-div}
\begin{cases}
\partial_t u + \dive(\tilde \b u) -\frac{1}{2} \sum_{i,j}\partial_{i}(a_{ij}\partial_j u)=0 &  \text{ in } (0,T) \times \R^d, \\
u\vert_{t=0}=u_0 & \text{ in } \R^d,
\end{cases}
\end{equation}
which is the Cauchy problem for the Fokker-Planck equation in divergence form. Requiring only integrability assumptions on $\tilde\b$ and $\a$ we can give the following definition of solutions to \eqref{eq:divFP}.
\begin{defn}\label{def:weak-solution}
Assume that 
\begin{equation}\label{eq:ass_para_sol_FP}
	\tilde\b\in L^1((0,T); L^2(\R^d;\R^d)), \qquad \a\in L^2((0,T); L^2(\R^d; \R^{d\times d})), \qquad u_0\in L^2(\R^d).
\end{equation}
A function $u\in L^{\infty}((0,T);L^2(\R^d))\cap L^2((0,T);H^1(\R^d))$ is a {\em parabolic solution} to \eqref{eq:divFP} if for any $\varphi\in C^\infty_c([0,T)\times\R^d)$ the following identity holds: 
	$$
	\int_0^T\int_{\R^d} u\bigg(\partial_t\varphi+ \tilde\b\cdot\nabla\varphi\bigg)-\frac{1}{2}\sum_{ij}	a_{ij}\partial_j u\partial_{i}\varphi \de x \de t+\int_{\R^d} u_0\varphi(0,\cdot) \de x=0.
	$$
\end{defn}
The following theorem holds.
\begin{thm}\label{thm:esistenza unicita fp div}
Let $p,q \in [2,\infty]$ be such that $1/p+1/q= 1/2$. Assume that: 
\begin{itemize}
    \item[(i)] $\tilde\b\in L^2((0,T); L^p(\R^d;\R^d))$;
    \item[(ii)] $\tilde\b$ satisfies the growth condition $\frac{\tilde\b}{1+|x|}:=\tilde\b_1+\tilde\b_2$ with
    \begin{equation}\label{growth condition b}
    \tilde\b_1\in L^1((0,T);L^1(\R^d;\R^d)),\qquad \tilde\b_2\in L^1((0,T);L^\infty(\R^d;\R^d));
    \end{equation}
    \item[(iii)] $(\dive \tilde\b)^{-}\in L^1((0,T);L^\infty(\R^d))$;
    \item[(iv)] $\a$ satisfies \ref{assu:bounded} and \ref{assu:elliptic}; 
    \item[(v)] $u_0\in L^2\cap L^q(\R^d)$. 
\end{itemize}
Then, there exists a unique parabolic solution $u$ to \eqref{eq:divFP}.
\end{thm}
\begin{proof}
Let $(\rho^\delta)_\delta$ be a standard family of mollifiers as in \eqref{def:moll} and let us define $\tilde\b^\delta:=\tilde\b*\rho^\delta$, $\a^\delta:=\a*\rho^\delta$, and $u_0^\delta=u_0*\rho^\delta$. Then, we consider the approximating problem
	\begin{equation}\label{eq:divFP-e}
	\begin{cases}
	\partial_t u^\delta+\dive(\tilde\b^\delta u^\delta)-\frac{1}{2} \sum_{i,j}\partial_{i}(a_{ij}^\delta \partial_j u^\delta)=0,\\
	u^\delta(0,\cdot)=u_0^\delta. 
	\end{cases}
	\end{equation}  
The equation \eqref{eq:divFP-e} is classically well-posed, so we can consider for any fixed $\delta>0$ its unique smooth solution $u^\delta$. We multiply the equation \eqref{eq:divFP-e} by $u^\delta$ and we integrate over $\R^d$, obtaining that 
\begin{equation}
\frac12\frac{\de}{\de t}\int_{\R^d}|u^\delta(t,x)|^2\de x+\frac12\int_{\R^d}|u^\delta(t,x)|^2\dive\tilde\b^\delta(t,x)\de x +\frac{1}{2} \sum_{i,j}\int_{\R^d}a_{ij}^\delta \partial_j u^\delta\partial_i u^\delta=0.
\end{equation}
Thus, we use the uniform ellipticity of $\a^\delta$ to obtain the inequality
\begin{equation}
\frac{\de}{\de t}\int_{\R^d}|u^\delta(t,x)|^2\de x+\alpha \int_{\R^d}|\nabla u^\delta(t,x)|^2\de x\leq \|(\dive\tilde\b^\delta(t,\cdot))^{-}\|_{\infty}\int_{\R^d}|u^\delta(t,x)|^2\de x,
\end{equation}
and by using Gronwall's lemma we obtain that
\begin{equation}
    \|u^\delta\|_{L^\infty L^2}+ \|u^\delta\|_{L^2H^1}\leq C(\alpha, \|(\dive\tilde\b)^{-}\|_{L^1 L^\infty})\|u_0\|_{L^2}.
\end{equation}
A standard compactness argument and the linearity of the equation \eqref{eq:divFP} guarantee that 
$$
u^\delta\weaktos u \in L^\infty L^2\cap L^2H^1,
$$
where $u$ is a weak solution to \eqref{eq:divFP}. We now move to the uniqueness: assume that $u$ is a weak solution to \eqref{eq:divFP} with $u_0=0$ and define $u^\delta:=u*\rho^\delta$. Then, $u^\delta$ solves the equation
\begin{equation}\label{eq:divFP-e+comm}
	\begin{cases}
	\partial_t u^\delta+\dive(\tilde\b u^\delta)-\frac{1}{2} \sum_{i,j}\partial_{i}(a_{ij} \partial_j u^\delta)=r^\delta+s_1^\delta,\\
	u^\delta(0,\cdot)=0,
	\end{cases}
	\end{equation}  
where $r^\delta$ and $s_1^\delta$ are the commutators defined in \eqref{eq:comm_def} and \eqref{def:s_1}. Consider a smooth function $\varphi$ such that 
$$
\varphi\geq 0,\hspace{0.3cm}\mathrm{supp} \ \varphi\subset B_2,\hspace{0.3cm}\varphi= 1 \mbox{ in }B_1,
$$
and define $\varphi_R(x)=\varphi\left(x/R\right)$.
Let $\beta_M\colon\R\to \R$ be an even function with the following properties: 
\begin{itemize}
    \item $\beta_M\in C^2\cap L^\infty(\R)$ with $\|\beta_M'\|_{\infty}< C M$ and $\|\beta_M''\|_\infty <C$ for some positive constant $C$ independent on $M$; 
    \item $\beta_M(z)=z^2$ for $|z|\le M$ for some positive $M$; 
    \item $0\le \beta_M(z) \le 2M^2$ for every $z \in \R$;
    \item $\beta_M(z) \le z^2$ for every $z \in \R$; 
    \item $\beta_M(z)=0 \iff z = 0$;
    \item $|\beta_M'(z)|\le C|z|$.
\end{itemize}

We multiply the equation \eqref{eq:divFP-e+comm} by $\beta_M'(u^\delta)\varphi_R$ and integrating on $(0,t)\times\R^d$ we obtain the following: for the term involving the vector field we have that
\begin{align*}
& \int_0^t\int_{\R^d}\dive(\tilde\b u^\delta)\beta_M'(u^\delta)\varphi_R\de x\de s \\ 
& = -\int_0^t \int_{\R^d}u^\delta \tilde\b \cdot\nabla \beta_M'(u^\delta)\varphi_R\de x\de s - \int_0^t \int_{\R^d} u^\delta \beta_M'(u^\delta)\tilde\b \cdot\nabla \varphi_R\de x\de s \\ 
& = -\int_0^t \int_{\R^d}u^\delta \beta_M''(u^\delta)\tilde\b \cdot\nabla u^\delta \varphi_R\de x\de s - \int_0^t \int_{\R^d} u^\delta \beta_M'(u^\delta)\tilde\b \cdot\nabla \varphi_R\de x\de s \\
& = - \iint_{\{|u^\delta|<M\}} \left( u^\delta \beta_M''(u^\delta)\tilde\b \cdot\nabla u^\delta \varphi_R + u^\delta \beta_M'(u^\delta)\tilde\b \cdot\nabla \varphi_R\right) \de x\de s \\
 & \qquad - \iint_{\{|u^\delta|>M\}} \left( u^\delta \beta_M''(u^\delta)\tilde\b \cdot\nabla u^\delta \varphi_R + u^\delta \beta_M'(u^\delta)\tilde\b \cdot\nabla \varphi_R\right) \de x\de s.
 \end{align*}
Now we use that $\beta_M(u^\delta)=|u^\delta|^2$ for $|u^\delta|\le M$, so
 \begin{align*}
 & = - 2\iint_{\{|u^\delta|<M\}} \left( u^\delta \tilde\b \cdot\nabla u^\delta \varphi_R + |u^\delta|^2 \tilde\b \cdot\nabla \varphi_R\right) \de x\de s \\
 & \qquad - \iint_{\{|u^\delta|>M\}} \left( u^\delta \beta_M''(u^\delta)\tilde\b \cdot\nabla u^\delta \varphi_R + u^\delta \beta_M'(u^\delta)\tilde\b \cdot\nabla \varphi_R\right) \de x\de s\\
 & = - \iint_{\{|u^\delta|<M\}} \left( \tilde\b \cdot\nabla \beta_M(u^\delta) \varphi_R + 2\beta_M(u^\delta) \tilde\b \cdot\nabla \varphi_R\right) \de x\de s \\
 & \qquad - \iint_{\{|u^\delta|>M\}} \left( u^\delta \beta_M''(u^\delta)\tilde\b \cdot\nabla u^\delta \varphi_R + u^\delta \beta_M'(u^\delta)\tilde\b \cdot\nabla \varphi_R\right) \de x\de s\\
  & = - \iint_{(0,t)\times \R^d} \tilde\b \cdot\nabla \beta_M(u^\delta) \varphi_R \de x\de s 
 - \iint_{\{|u^\delta|<M\}} 2\beta_M(u^\delta) \tilde\b \cdot\nabla \varphi_R \de x\de s\\
 & \quad - \iint_{\{|u^\delta|>M\}} \left( u^\delta \beta_M''(u^\delta)\tilde\b \cdot\nabla u^\delta \varphi_R + u^\delta \beta_M'(u^\delta)\tilde\b \cdot\nabla \varphi_R - \tilde\b \cdot\nabla \beta_M(u^\delta) \varphi_R \right) \de x\de s.
 \end{align*}
Finally, by integrating by parts we get that
 \begin{align*}
  & = \iint_{(0,t)\times \R^d} \left( \dive\tilde\b \beta_M(u^\delta) \varphi_R + \beta_M(u^\delta) \tilde\b \cdot \nabla \varphi_R \right)\de x\de s 
 \\
 & \qquad - \iint_{\{|u^\delta|<M\}} 2\beta_M(u^\delta) \tilde\b \cdot\nabla \varphi_R \de x\de s\\
 & \qquad - \iint_{\{|u^\delta|>M\}} \left( u^\delta \beta_M''(u^\delta)\tilde\b \cdot\nabla u^\delta \varphi_R + u^\delta \beta_M'(u^\delta)\tilde\b \cdot\nabla \varphi_R - \tilde\b \cdot\nabla \beta_M(u^\delta) \varphi_R \right) \de x\de s\\
 & = \iint_{(0,t)\times \R^d} \dive\tilde\b \beta_M(u^\delta) \varphi_R \de x\de s +\iint_{\{|u^\delta>M\}} \beta_M(u^\delta) \tilde\b \cdot \nabla \varphi_R \de x\de s
 \\
 & \qquad - \iint_{\{|u^\delta|<M\}} \beta_M(u^\delta) \tilde\b \cdot\nabla \varphi_R \de x\de s\\
 & \qquad - \iint_{\{|u^\delta|>M\}} \left( u^\delta \beta_M''(u^\delta)\tilde\b \cdot\nabla u^\delta \varphi_R + u^\delta \beta_M'(u^\delta)\tilde\b \cdot\nabla \varphi_R - \tilde\b \cdot\nabla \beta_M(u^\delta) \varphi_R \right) \de x\de s.
 %
\end{align*}
On the other hand, for the term involving the diffusion we have
\begin{align}
\int_0^t\int_{\R^d}\partial_{i}(a_{ij}&\partial_j u^\delta)\beta_M'(u^\delta)\varphi_R\de x\de s\nonumber\\
&=-\int_0^t\int_{\R^d}a_{ij}\partial_j u^\delta\partial_iu^\delta\beta_M''(u^\delta)\varphi_R\de x \de s-\int_0^t\int_{\R^d}a_{ij}\partial_j u^\delta\beta_M'(u^\delta)\partial_i\varphi_R\de x \de s\nonumber\\
&=-\int_0^t\int_{\R^d}a_{ij}\partial_j u^\delta\partial_iu^\delta\beta_M''(u^\delta)\varphi_R\de x \de s
-\int_0^t\int_{\R^d}a_{ij}\partial_j \beta_M(u^\delta)\partial_i\varphi_R\de x \de s.\label{commento lbl}
\end{align}
Then, we use the uniform ellipticity of $\a$, i.e.
\begin{equation}
2\alpha\iint_{\{|u^\delta|<M\}}|\nabla u^\delta|^2\varphi_R\de x \de s\leq \iint_{\{|u^\delta|<M\}}a_{ij}\partial_i u^\delta\partial_j u^\delta\beta_M''(u^\delta)\varphi_R\de x \de s
\end{equation}
to obtain the following inequality 
\begin{align}
\int_{\R^d}\beta_M(u^\delta)\varphi_R\de x&+ \alpha\iint_{\{|u^\delta|<M\}}|\nabla u^\delta|^2\varphi_R\de x \de s\nonumber\\
& \le -  \iint_{(0,t)\times \R^d} \dive\tilde\b \beta_M(u^\delta) \varphi_R \de x\de s  -\iint_{\{|u^\delta>M\}} \beta_M(u^\delta) \tilde\b \cdot \nabla \varphi_R \de x\de s
 \nonumber\\
 & + \iint_{\{|u^\delta|<M\}} \beta_M(u^\delta) \tilde\b \cdot\nabla \varphi_R \de x\de s\nonumber\\
 & + \iint_{\{|u^\delta|>M\}} \left( u^\delta \beta_M''(u^\delta)\tilde\b \cdot\nabla u^\delta \varphi_R + u^\delta \beta_M'(u^\delta)\tilde\b \cdot\nabla \varphi_R - \tilde\b \cdot\nabla \beta_M(u^\delta) \varphi_R \right) \de x\de s \nonumber\\ 
 & -\frac12 \iint_{\{|u^\delta|>M\}}a_{ij}\partial_j u^\delta\partial_iu^\delta\beta_M''(u^\delta)\varphi_R\de x \de s -\frac12 \int_0^t\int_{\R^d}a_{ij}\partial_j \beta_M(u^\delta)\partial_i\varphi_R\de x \de s \nonumber\\
 & + \int_0^t\int_{\R^d}(r^\delta+s^\delta)\beta_M'(u^\delta)\varphi_R\de x\de s. \nonumber
 \end{align}
Then, we exploit the fact that we need to account only for the negative part of the divergence in the term on the right, and using trivial estimates we obtain that
 \begin{align}
 \int_{\R^d}\beta_M&(u^\delta)\varphi_R\de x \le \iint_{(0,t)\times \R^d} (\dive\tilde\b)^- \beta_M(u^\delta) \varphi_R \de x\de s +  \iint_{(0,t)\times \R^d} \beta_M(u^\delta) |\tilde\b||\nabla \varphi_R| \de x\de s\nonumber \\
 & +  \iint_{\{|u^\delta|>M\}} \left(|u^\delta\tilde\b| |\nabla u^\delta| | \beta_M''(u^\delta)|+|\tilde\b||\nabla u^\delta||\beta_M'(u^\delta)|+\frac12|\a||\nabla u^\delta|^2|\beta_M''(u^\delta)|\right)\varphi_R \de x \de s  \nonumber\\
 & +  \iint_{\{|u^\delta|>M\}}|u^\delta \tilde\b||\beta_M'(u^\delta)||\nabla\varphi_R|\de x \de s+\frac12\int_0^t\int_{\R^d}|\a||\beta_M'(u^\delta)||\nabla u^\delta||\nabla\varphi_R|\de x \de s \nonumber\\ 
 & + \int_0^t\int_{\R^d}(r^\delta+s^\delta)\beta_M'(u^\delta)\varphi_R\de x\de s.\label{stime cacca}
\end{align}
Let us estimate by Young 
\begin{align*}
& \iint_{\{|u^\delta|>M\}} \left(|u^\delta\tilde\b| |\nabla u^\delta| | \beta_M''(u^\delta)|+|\tilde\b||\nabla u^\delta||\beta_M'(u^\delta)|+\frac12|\a||\nabla u^\delta|^2|\beta_M''(u^\delta)|\right)\varphi_R \de x \de s  \\
& \le \|\beta_M''\|_\infty  \iint_{\{|u^\delta|>M\}}|u^\delta\tilde\b| |\nabla u^\delta| \de x \de s+C\iint_{\{|u^\delta|>M\}}|\tilde\b||u^\delta||\nabla u^\delta|\de x\de s\\
&+\frac12\|\a\|_\infty\|\beta_M''\|_\infty  \iint_{\{|u^\delta|>M\}}|\nabla u^\delta|^2 \de x \de s\\
& \le C( \|\beta_M''\|_\infty,\|\a\|_\infty) \iint_{\{|u^\delta|>M\}}(|u^\delta\tilde\b|^2 +|\nabla u^\delta|^2 ) \de x \de s. 
\end{align*}
By Chebishev's inequality we can estimate the super-level as follows
\begin{equation}\label{eq:chebishev}
\mathscr{L}^{d+1}\big(\{ (t,x)\in(0,T)\times \R^d:|u^\delta(t,x)|>M \} \big)\leq \frac{\|u^\delta\|^2_{L^2_{t,x}}}{M^2}\leq \frac{C(T,\|u_0\|_{L^2})}{M^2},
\end{equation}
which is small in $M$ uniformly in $\delta$. Then, since the quantity $(|u^\delta\tilde\b|^2 +|\nabla u^\delta|^2 )\in L^1_{t,x}$ uniformly in $\delta$ by the assumptions, we use the equi-integrability of the integrand to obtain the following estimate
\begin{equation}
    \iint_{\{|u^\delta|>M\}} \left(|u^\delta\tilde\b| |\nabla u^\delta| | \beta_M''(u^\delta)|+|\tilde\b||\nabla u^\delta||\beta_M'(u^\delta)|+\frac12|\a||\nabla u^\delta|^2|\beta_M''(u^\delta)|\right)\varphi_R \de x \de s \leq F(M),
\end{equation}
for some function $F$, depending on $\|\beta_M''\|_{\infty},\|\a\|_\infty,\|u_0\|_{L^2}, T$, such that $F(M)\to 0$ as $M\to \infty$.
On the other hand, for the third line in \eqref{stime cacca} we proceed as follows: for the term involving the vector field we use H\"older inequality and \eqref{eq:chebishev} to obtain the estimate
\begin{align*}
\iint_{\{|u^\delta|>M\}}|u^\delta \tilde\b||\beta_M'(u^\delta)||\nabla\varphi_R|\de x \de s &\leq\frac{\|\beta_M'\|_\infty}{R}\|u^\delta\tilde\b\|_{L^2}\frac{C(T,\|u_0\|_{L^2})}{M}:=\frac{C}{R}.
\end{align*}
For the term involving the diffusion we use the definition of $\beta_M$ to obtain
\begin{align*}
\int_0^t\int_{\R^d}|\a||\beta_M'(u^\delta)||\nabla u^\delta||\nabla\varphi_R|\de x\de s &
\leq \frac{C\|\a\|_\infty}{R} \int_0^t\int_{\R^d}|\nabla u^\delta||u^\delta|\de x \de s\\
&\leq \frac{C\|\a\|_\infty}{R}\|\nabla u\|_{L^2L^2}\|u\|_{L^\infty L^2}\\
&\le \frac{C}{R},
\end{align*}
where in the penultimate inequality we used H\"older's inequality, and the constant $C$ depends on $\|\a\|_{L^\infty}, \|u_0\|_{L^2}, T$.
Summarizing, we rewrite \eqref{stime cacca} as follows
\begin{align*}
\int_{\R^d}\beta_M(u^\delta)\varphi_R\de x&+ \alpha\iint_{\{|u^\delta|<M\}}|\nabla u^\delta|^2\varphi_R\de x \de s\\
& \le \iint_{(0,t)\times \R^d} (\dive\tilde\b)^{-} \beta_M(u^\delta) \varphi_R \de x\de s + \iint_{(0,t)\times \R^d} \beta_M(u^\delta) |\tilde\b||\nabla \varphi_R| \de x\de s\\
& \quad+F(M)+\frac{C}{R} + \int_0^t\int_{\R^d}(r^\delta+s^\delta)\beta_M'(u^\delta)\varphi_R\de x\de s.
\end{align*}
Since $\beta_M'(u^\delta)\varphi_R \to \beta_M'(u)\varphi_R$ strongly in $L^2((0,T); H^1_0(B_{2R}))$, we can use Lemma \ref{lem:conv_comm2} and Lemma \ref{lem:conv_comm5} and, by letting $\delta\to 0$, we get that
\begin{align*}
\int_{\R^d}\beta_M(u)\varphi_R\de x &\le \iint_{(0,t)\times \R^d} (\dive\tilde\b)^{-} \beta_M(u) \varphi_R \de x\de s + \iint_{(0,t)\times \R^d} \beta_M(u) |\tilde\b||\nabla \varphi_R| \de x\de s + \frac{C}{R}\\
&\le \int_0^t\|(\dive\tilde\b(s,\cdot))^-\|_\infty\int_{\R^d}\beta_M(u)\varphi_R\de x \de s+2M^2\int_0^T\int_{|x|>R}|\tilde\b_1(s,x)|\de x\de s\\
&\quad +\int_0^T\|\tilde\b_2(s,\cdot)\|_\infty\int_{|x|>R}|u(s,x)|^2\de x \de s+F(M)+\frac{C}{R} \\
&:= \int_0^t\|(\dive\tilde\b(s,\cdot))^-\|_\infty\int_{\R^d}\beta_M(u)\varphi_R\de x \de s + H(R,M)+F(M), 
\end{align*}
where $\tilde\b_1$ and $\tilde\b_2$ are as in \eqref{growth condition b}, and the function $H(R,M)$ goes to $0$ as $R\to\infty$ for any fixed $M>0$. Notice that in the last inequality we used the definition of $\varphi_R$, the growth assumption on $\tilde\b$, and the properties of $\beta_M$. 
Finally, we apply Gronwall's inequality 
\begin{equation*}
\int_{\R^d}\beta_M(|u(t,x)|)\varphi_R\de x  \leq [H(R,M)+F(M)]\exp\left(\|(\dive\tilde\b)^-\|_{L^1L^\infty}\right),
\end{equation*}
and by letting $R\to\infty$, we obtain that 
\begin{equation}\label{final_gronwall}
    \int_{\R^d}\beta_M(|u(t,x)|)\de x\leq F(M)\exp\left(\|(\dive\tilde\b)^-\|_{L^1L^\infty}\right),
\end{equation}
for any $M>0$. By Lebesgue's dominated convergence Theorem 
$$
\int_{\R^d}\beta_M(|u(t,x)|)\de x\to \int_{\R^d}|u(t,x)|^2\de x ,
$$
as $M\to \infty$, and together with \eqref{final_gronwall}, this implies that $u(t,x)=0$ for a.e. $(t,x)\in(0,T)\times\R^d$.
\end{proof}

\begin{rem}
    An example of a function $\beta_M$ satisfying the assumptions of Theorem \ref{thm:esistenza unicita fp div} is given by a regularized version of the following function:
    \begin{equation}
    \beta_M(z):=
        \begin{cases}
            z^2\qquad &\mbox{ if }|z|\leq M,\\
            -\frac{z^2}{2}+2M z-\frac{M^2}{2} &\mbox{ if } M<|z|<2M,\\
            \frac72 M^2 &\mbox{ if }|z|>2M.
        \end{cases}
    \end{equation}
\end{rem}

\subsubsection{Comparison with \cite{LBL_CPDE}}\label{sss:comparison} We now compare the uniqueness result of 
Theorem \ref{thm:esistenza unicita fp div} with Proposition 4 of \cite{LBL_CPDE}, which we report here for completeness.
\begin{prop}\label{prop:lebris-lions}
    Assume that the matrix $\a:=\bsigma\bsigma^T$ is uniformly positive definite and that
    \begin{equation}
    \begin{cases}
        \tilde\b\in L^2((0,T);L^2_\mathrm{loc}(\R^d;\R^d)),\qquad \frac{\tilde\b}{1+|x|}\in L^1((0,T);L^1+L^\infty(\R^d;\R^d)),\\
        \dive\tilde\b \in L^1((0,T);L^\infty(\R^d)),\\
        \bsigma\in L^\infty((0,T);L^\infty_\mathrm{loc}(\R^d;\R^{d\times d})),\qquad \frac{\bsigma}{1+|x|}\in L^2((0,T);L^2+L^\infty(\R^d; \R^{d\times d})).
    \end{cases}
    \end{equation}
Then, for each initial condition in $L^1\cap L^\infty(\R^d)$, equation \eqref{eq:divFP} has a unique solution in the space
\begin{equation}
    \{u\in L^\infty((0,T);L^1\cap L^\infty(\R^d)), \,\,u\in L^2((0,T);H^1(\R^d)),\,\, \bsigma^T \nabla u \in L^2((0,T);L^2(\R^d))\}.
\end{equation}
\end{prop}

Note that if we consider the assumptions of Theorem \ref{thm:esistenza unicita fp div} with $p=2,q=\infty$ LeBris-Lions' result is not immediately recovered because of local integrability assumptions. However, by making the following changes to the proof of Theorem \ref{thm:esistenza unicita fp div}, we can prove a slightly different result. 
\begin{cor}\label{cor:esistenza unicita fp div}
    Assume that the matrix $\a:=\bsigma\bsigma^T$ is uniformly positive definite and that
    \begin{equation}
    \begin{cases}
        \tilde\b\in L^2((0,T);L^2(\R^d;\R^d)),\qquad \frac{\tilde\b}{1+|x|}\in L^1((0,T);L^1+L^\infty(\R^d;\R^d)),\\
        (\dive\tilde\b)^- \in L^1((0,T);L^\infty(\R^d)),\\
        \bsigma\in L^\infty((0,T);L^\infty_\mathrm{loc}(\R^d;\R^{d\times d})),\qquad \frac{\bsigma}{1+|x|}\in L^2((0,T);L^2+L^\infty(\R^d; \R^{d\times d})).
    \end{cases}
    \end{equation}
Then, for each initial condition in $L^1\cap L^\infty(\R^d)$, equation \eqref{eq:divFP} has a unique solution in the space
\begin{equation}
    \{u\in L^\infty((0,T);L^1\cap L^\infty(\R^d)), \,\,u\in L^2((0,T);H^1(\R^d)),\,\, \bsigma^T \nabla u \in L^2((0,T);L^2(\R^d))\}.
\end{equation}
\end{cor}
\begin{proof}
Since the argument is very similar to that of the Theorem \ref{thm:esistenza unicita fp div} we only sketch the proof. First, since we are dealing with bounded solutions, we can consider the function $\beta(z)=|z|^2$ and all the integrals on the set $\{|u^\delta|>M\}$ are identically zero for $M$ large enough. This implies that we have to use a Gronwall type argument on the following estimate
\begin{align}
\int_{\R^d}\beta(u^\delta)\varphi_R\de x&+ \alpha\iint_{\{|u^\delta|<M\}}|\nabla u^\delta|^2\varphi_R\de x \de s
\le -  \iint_{(0,t)\times \R^d} \dive\tilde\b \beta(u^\delta) \varphi_R \de x\de s \nonumber\\
&+\int_0^t\int_{\R^d}(r^\delta+s^\delta)\beta'(u^\delta)\varphi_R\de x\de s+  \iint_{(0,t)\times \R^d} \beta(u^\delta) |\tilde \b||\nabla \varphi_R| \de x\de s\nonumber  \\
& +\frac12\int_0^t\int_{\R^d}a_{ij}\partial_j\beta(u^\delta)\partial_i\varphi_R\de x \de s.\label{caccacacca}
\end{align}
The local integrability assumptions together with the bound on $(\dive \tilde\b)^-$ are sufficient for the convergence of the commutators. On the other hand, we use that the matrix $\a$ is representable as $a_{ij} = \sigma_{ik}\sigma_{jk}$, to estimate the last term in \eqref{caccacacca} as follows
\begin{align*}
\int_0^t\int_{\R^d}a_{ij}\partial_j\beta(u^\delta)\partial_i\varphi_R\de x \de s & = \int_0^t\int_{\R^d}\sigma_{ik}\sigma_{jk}\partial_ju^\delta \beta'(u^\delta)\partial_i\varphi_R\de x \de s\\
&\le \frac12 \int_0^t\int_{\R^d}|\sigma_{jk}\partial_ju^\delta|^2 |\partial_i\varphi_R|\de x \de s\\
&+\frac12 \int_0^t\int_{\R^d}|\sigma_{ik}|^2|\beta'(u^\delta)|^2|\partial_i\varphi_R|\de x \de s\\
&\le \frac{C}{R}\int_0^t\int_{\R^d}|\sigma_{jk}\partial_ju^\delta|^2\de x \de s\\
&+C\|u_0\|_\infty^2\int_0^t\int_{|x|>R}|\sigma_{ik}^1|^2 \de x \de s+\int_0^t\|\sigma_{ik}^2(s,\cdot)\|_\infty\int_{|x|>R}|u^\delta|^2\de x\de s,
\end{align*}
which goes to $0$ as $R\to\infty$. The conclusion follows along the same lines as in the proof of Theorem \ref{thm:esistenza unicita fp div}.
\end{proof}

Notice that our corollary works under assumptions of global integrability on the vector field $\tilde\b$, but we only assume that the negative part of the divergence is bounded. Remarkably, we can use this weaker assumption thanks to the convergence of the commutators in $L^2_tH^{-1}_x$. Assuming a bound on the full divergence we would then be able to consider the case of locally integrable vector fields, as done by LeBris-Lions, relying on the convergence in $L^1$ of the commutator $r^\delta$ (see Lemma \ref{lem:conv_comm}) and finding the very same proof.  

\subsection{The Fokker-Planck equation} 

We start by giving the following definition.
\begin{defn}[Distributional solution]\label{def:weak_sol_fp}
Assume that 
\begin{equation}\label{eq:ass_distro_sol_FP}
\b\in L^1((0,T); L^p(\R^d)), \qquad \a\in L^1((0,T); L^p(\R^d; \R^{d\times d})), \qquad u_0\in L^q(\R^d)
\end{equation}
are given, with $p,q$ satisfying $\sfrac{1}{p} +\sfrac{1}{q} \leq 1$. A function $u\in L^{\infty}((0,T);L^q(\R^d))$ is a {\em distributional solution} to \eqref{eq:fp} if for any $\varphi\in C^\infty_c([0,T)\times\R^d)$ the following identity holds: 
	$$
	\int_0^T\int_{\R^d} u\bigg(\partial_t\varphi+\b\cdot\nabla\varphi+\frac{1}{2}\sum_{ij}	a_{ij}\partial_{ij}\varphi\bigg) \de x \de t+\int_{\R^d} u_0\varphi(0,\cdot) \de x=0.
	$$
\end{defn}
Notice that in the definition of distributional solutions the assumption that $p,q$ satisfy $\sfrac{1}{p} +\sfrac{1}{q} \leq 1$ is the minimum requirement we need in order to have $u \b, u\a\in L^1$ so that the definition makes sense.

We now recall the following definition.
\begin{defn}\label{def:equiint}
Let $\Omega$ be an open subset of $\R^d$. We say that a bounded family $\{\varphi_i\}_{i\in I}$ is equi-integrable if the following two conditions hold:
\begin{itemize}
    \item[$(i)$] For any $\e>0$ there exists a borel set $A\subset\Omega$ with finite measure such that
    $$
    \int_{\Omega\setminus A} |\varphi_i(x)|\de x\leq\e,
    $$
    for any $i\in I$.
    \item[$(ii)$] For any $\e>0$ there exists $\delta>0$ such that, for any Borel set $E\subset \Omega$ with $\mathscr{L}^d(E)\leq\delta$, there holds
    $$
    \int_E|\varphi_i(x)|\de x\leq \e,
    $$
    for any $i\in I$.
\end{itemize}
\end{defn}
Before to prove the existence of distributional solutions we recall the following characterization of equi-integrability, see \cite{BC} (or alternatively \cite{BCC} for a proof on the $d$-dimensional torus $\T^d$).
\begin{lem}\label{lem:equi-integrability}
Consider a family $\{\varphi_i\}_{i\in I}\subset L^1(\Omega)$ which is bounded in $L^1(\Omega)$. Then this family satisfies condition $(ii)$ of Definition \ref{def:equiint} if and only if for every $\e>0$, there exists a constant $C_\e>0$ 
such that for every $i\in I$ we can write
$$
\varphi_i=\varphi_i^1+\varphi_i^2,
$$
with
\begin{equation}\label{decomposizione equiint}
    \|\varphi_i^1\|_{L^1(\Omega)}\le \e,
    \qquad\|\varphi_i^2\|_{L^2(\Omega)}\le C_\e, \,\,\mbox{for all } i\in I. 
\end{equation}
\end{lem}
The proof of existence of distributional solutions immediately follows from a classical a priori estimate. 
\begin{proposition}[Existence of distributional solutions]\label{prop:existence_weak_sol_2}
Let $\b,\a$ and $u_0$ be as in \eqref{eq:ass_distro_sol_FP} with $p,q$ such that $\sfrac{1}{p} +\sfrac{1}{q} \leq 1$ and let \ref{assu:dive_a_and_b} and \ref{assu:elliptic} hold and assume that $(\dive \tilde \b)^- \in L^1((0,T);L^\infty(\R^d))$. Then there exists a distributional solution $u\in L^{\infty}((0,T);L^q(\R^d))$ to \eqref{eq:fp}.
\end{proposition}
\begin{proof}
Let $(\rho^\delta)_\delta$ be a standard family of mollifiers as in \eqref{def:moll} and let us define $\b^\delta:=\b*\rho^\delta$, $\a^\delta:=\a*\rho^\delta$, and $u_0^\delta=u_0*\rho^\delta$. Then, we consider the approximating problem
	\begin{equation}\label{eq:ad_bae}
	\begin{cases}
	\partial_t u^\delta+\dive(\b^\delta u^\delta)-\frac{1}{2} \sum_{i,j}\partial_{ij}(a_{ij}^\delta u^\delta)=0,\\
	u^\delta(0,\cdot)=u_0^\delta. 
	\end{cases}
	\end{equation}
The existence of a unique solution of \eqref{eq:ad_bae} follows from classical theory \cite{EV}. 
It is readily checked that the sequence $u^\delta$ is equi-bounded in $L^\infty((0,T);L^q(\T^d))$. Indeed, from the smoothness of the objects in \eqref{eq:ad_bae}, we can rewrite it in divergence form \begin{equation}\label{eq:ad_bae-div}
\partial_t u^\delta + \dive(\tilde{\b}^\delta u^\delta) -\frac{1}{2} \sum_{i,j}\partial_{i}(a_{ij}^\delta\partial_j u^\delta)=0,
\end{equation}
and then we can multiply the equation by $\beta'(u^\delta)$, where $\beta \colon \R \to \R$ is a smooth, convex function: by an easy application of the chain rule and by integrating in space we obtain (for fixed $t$)
\begin{align*}
\int_{\R^d}\dive(\tilde{\b}^\delta u^\delta)\beta'(u^\delta)\de x&=\int_{\R^d} u^\delta \beta'(u^\delta)\dive\tilde{\b}^\delta\de x+\int_{\R^d}\tilde{\b}^\delta\cdot\nabla\beta(u^\delta)\de x\\
&=\int_{\R^d} u^\delta\beta'(u^\delta) \dive\tilde{\b}^\delta\de x + \int_{\R^d}\dive(\tilde{\b}^\delta \beta(u^\delta))\de x-\int_{\R^d}\dive\tilde{\b}^\delta\beta(u^\delta)\de x,
\end{align*}
and
\begin{equation*}
\int_{\R^d}\partial_i(a_{ij}^\delta\partial_j u^\delta)\beta'(u^\delta)\de x =-\int_{\R^d} a_{ij}^\delta\beta''(u^\delta)\partial_i u^\delta\partial_j u^\delta\de x.
\end{equation*}
By considering a sequence of smooth, convex functions, uniformly convergent to $\beta(s) = |s|^q$, for $1 < q < \infty$, the following uniform bounds on the $L^q$-norm of the solutions $u^\delta$:
\begin{align}
\frac{\de}{\de t}\|u^\delta(t,\cdot)\|_{L^q}^q&\leq -(q-1)\int_{\R^d}\dive\tilde{\b}^\delta(t,x)|u^\delta(t,x)|^q\de x-\alpha\int_{\R^d}|\nabla u^\delta(t,x)|^2\beta''(u^\delta(t,x))\de x \label{est:matr-1}\\
&\leq (q-1)\|(\dive\tilde{\b}^\delta)^-(t,\cdot)\|_{\infty}\|u^\delta(t,\cdot)\|_{L^q}^q,\label{est:matr-2}
\end{align}
where in \eqref{est:matr-1} we have used that, by standard properties of mollifiers, $\a^\delta$ is uniformly elliptic with the same constant $\alpha$ as in \ref{assu:elliptic}. Notice that $\|(\dive\tilde{\b}^\delta)^-\|_{L^1L^\infty}$ is uniformly bounded (in $\delta$) by \ref{assu:dive_a_and_b}. The conclusion now follows from an application of Gronwall's lemma: we have that
\begin{equation}\label{eq:diss_norm}
\sup_{t\in(0,T)}\|u^\delta(t,\cdot)\|_{L^q(\R^d)}\leq\|u_0\|_{L^q(\R^d)}e^{C(q)\|(\dive\tilde \b)^-\|_{L^1L^\infty}},
\end{equation}
with $C(q)=\frac{q-1}{q}$. Thus, for $q>1$ by standard compactness arguments, we can extract a subsequence which converges weakly-* to a function $u\in L^\infty((0,T);L^q(\R^d))$ and it is immediate to deduce that $u$ is a distributional solution of \eqref{eq:fp} because of the linearity of the equation. For $q=\infty$, the estimate \eqref{eq:diss_norm} still holds for every $\delta>0$: we send $q\to \infty$ in \eqref{eq:diss_norm} and then we can conclude as in the previous case. For the case $q=1$, 
the estimate is not sufficient to obtain weak compactness in $L^1$, as we need to show the equi-integrability of the family $(u^\delta)_{\delta>0}$. We first discuss the equi-integrability on small sets: let $u_0\in L^1(\R^d)$ and define $u_0^\delta$ as above. Since $u_0^\delta \to u_0$ strongly in $L^1(\R^d)$, it converges weakly and thus it is equi-integrable. Then, we apply Lemma \ref{lem:equi-integrability} to the family $u_0^\delta$: let $\e>0$ be fixed, there exists a constant $C_\e>0$ such that
$$
u_0^\delta=u_{0,1}^\delta+u_{0,2}^\delta,
$$
with
$$
\|u_{0,1}^\delta\|_{L^1}\leq \e,\qquad \|u_{0,2}^\delta\|_{L^2}\leq C_\e.
$$
We define $u_1^\delta$ and $u_2^\delta$ to be the solutions of \eqref{eq:ad_bae} with initial datum, respectively, $u_{0,1}^\delta$ and $u_{0,2}^\delta$. Thus, the estimate in \eqref{eq:diss_norm} with $q=1,2$ implies that
\begin{align}
\sup_{t\in(0,T)}\|u^\delta_1(t,\cdot)\|_{L^1(\R^d)}&\leq\e e^{C\|(\dive\tilde \b)^-\|_{L^1L^\infty}},\\
\sup_{t\in(0,T)}\|u^\delta_2(t,\cdot)\|_{L^2(\R^d)}&\leq C_\e e^{C\|(\dive\tilde \b)^-\|_{L^1L^\infty}}.
\end{align}
The linearity of the equations implies that the function $u^\delta=u_1^\delta+u_2^\delta$ is the unique solution of \eqref{eq:ad_bae} with initial datum $u_0^\delta$ and it admits a decomposition as in \eqref{decomposizione equiint}. Then, we can infer that the sequence $u^\delta$ satisfies $(ii)$ of Definition \ref{def:equiint}.

    We now verify the condition on the tails, i.e. $(i)$ in Definition \ref{def:equiint}: let $r,R>0$ such that $2r<R$ and consider a positive test function $\psi_r^R$ such that
    \begin{equation}
    \begin{cases}
        0 \qquad &\mbox{if }0<|x|<r,\\
        1 &\mbox{if }2r<|x|<R,\\
        0 &\mbox{if }|x|>R,
    \end{cases}
    \end{equation}
    such that $0\leq \psi_r^R\leq 1$ and 
    \begin{equation}
        |\nabla\psi_r^R|\leq\frac{C}{r},\qquad |\nabla^2\psi_r^R|\leq\frac{C}{r^2}.
    \end{equation}
    Then consider a smooth convex function $\beta_\e$ converging uniformly to $\beta(s)=|s|$ (for example we can take $\beta_\e(s)=\sqrt{s^2+\e^2}-\e$) and multiply the equation \eqref{eq:ad_bae-div} by $\beta'_\e(|u^\delta(t,x)|)\psi_r^R(x)$: integrating by parts we get that
    \begin{align*}
    \frac{d}{dt} \int_{\R^d} \beta_\e(|u^\delta(t,x)|)\psi_r^R (x)\, \de x  &+ \int_{\R^d} \dive \b^\delta(t,x)\big(u^\delta(t,x)  \beta_\e'(|u^\delta(t,x)|)- \beta_\e(|u^\delta(t,x)|)\big)\psi_r^R(x) \, \de x \\
    &-\int_{\R^d}\beta_\e(u^\delta(t,x))\b(t,x)\cdot\nabla \psi_r^R(x)\,\de x\\
    &+\frac{1}{2}\sum_{ij} \int_{\R^d} a_{ij}(t,x)\partial_j u^\delta(t,x) \partial_i u^\delta(t,x) \beta_\e^{\prime\prime}(|u^\delta(t,x)|) \psi_r^R(x)\, \de x \\
    &+\frac{1}{2}\sum_{ij} \int_{\R^d} a_{ij}(t,x)\partial_j \beta_\e(|u^\delta(t,x)|) \partial_i\psi_r^R(x)\, \de x= 0.      
    \end{align*}
    Then we estimate the fourth term above using the convexity of $\beta_\e$ together with \ref{assu:elliptic}, while we do a further integration by parts in the last term to obtain that
    \begin{align*}
    \frac{d}{dt} \int_{\R^d} \beta_\e(|u^\delta(t,x)|)\psi_r^R (x)\, \de x  &\leq \int_{\R^d} \dive \b^\delta(t,x)\big(u^\delta(t,x)  \beta_\e'(|u^\delta(t,x)|)- \beta_\e(|u^\delta(t,x)|)\big)\psi_r^R(x) \, \de x \\
    &+\int_{\R^d}|\beta_\e(u^\delta(t,x))||\b(t,x)||\nabla \psi_r^R(x)|\,\de x\\
    &+\frac{1}{2}\sum_{ij} \int_{\R^d} |\partial_j a_{ij}(t,x)| \beta_\e(|u^\delta(t,x)|) |\partial_i\psi_r^R(x)|\, \de x\\
    &+\frac{1}{2}\sum_{ij} \int_{\R^d} |a_{ij}(t,x)| \beta_\e(|u^\delta(t,x)|) |\partial_{ij}\psi_r^R(x)|\, \de x.
    \end{align*}
    Let $t\in (0,T)$ be fixed, we integrate on $(0,t)$ and by letting $\e\to 0$ we have that
    $$
    u^\delta(t,x)  \beta_\e'(|u^\delta(t,x)|)- \beta_\e(|u^\delta(t,x)|)\to 0,
    $$
    leading to the inequality
    \begin{align*}
    \int_{\R^d} |u^\delta(t,x)|\psi_r^R (x)\, \de x  &\leq \int_{\R^d} |u^\delta_0(x)|\psi_r^R (x)\, \de x+\int_{\R^d}|u^\delta(t,x)||\b(t,x)||\nabla \psi_r^R(x)|\,\de x\\
    &+\frac{1}{2}\sum_{ij} \int_{\R^d} |\partial_j a_{ij}(t,x)||u^\delta(t,x)| |\partial_i\psi_r^R(x)|\, \de x\\
    &+\frac{1}{2}\sum_{ij} \int_{\R^d} |a_{ij}(t,x)| |u^\delta(t,x)| |\partial_{ij}\psi_r^R(x)|\, \de x.
    \end{align*}
    Finally, we use the definition of $\psi_r^R$ and by letting $R\to \infty$ we obtain
    \begin{align*}
    \int_{|x|>r} |u^\delta(t,x)|\, \de x  &\leq \int_{|x|>r} |u^\delta_0(x)|\, \de x+C\bigg(\frac{\|\b\|_\infty+\sum_{ij}\|\partial_j a_{ij}\|_\infty}{r}\bigg)\int_{|x|>r}|u^\delta(t,x)|\,\de x\\
    &+C\frac{\sum_{ij}\|a_{ij}\|_\infty}{r^2}\int_{|x|>r}|u^\delta(t,x)|\,\de x.
    \end{align*}
    Note that the constant $C$ above does not depend on $\delta$ and $t$. The conclusion follows by using the equi-integrability of $u_0^\delta$ and the uniform bound on $\|u^\delta\|_{L^\infty L^1}$.
    So far we have shown that the sequence $\{u^\delta\}_{\delta>0}\subset L^1((0,T)\times\R^d)$ is uniformly integrable. This implies that there exists a function $u\in L^1((0,T)\times\R^d)$ such that, up to a subsequence, the following convergence holds
    $$
    u^\delta\weakto u \qquad\mbox{in }L^1((0,T)\times\R^d).
    $$
    In order to conclude we need to check that $u^\delta\weaktos u$ in $L^\infty L^1$. Note that the sequence $u^\delta$ satisfies the inequality \eqref{eq:diss_norm} with $q=1$, which implies that
    $$
    u^\delta\weaktos u\qquad \mbox{in }L^\infty((0,T);\mathcal{M}(\R^d)).
    $$
    Then, for any $\psi\in C^\infty_c([0,T))$ and $\phi\in C_0(\R^d)$ we have that
\begin{equation}
\int_0^T\psi(t)\int_{\R^d}\phi(x)\de u_t(x)\de t=\int_0^T\psi(t)\int_{\R^d}\phi(x) u(t,x)\de x\de t.
\end{equation}
We can conclude that for a.e. $t\in (0,T)$ we have that
\begin{equation}\label{uguale limite L1}
\int_{\R^d}\phi(x)\de u_t(x)=\int_{\R^d}\phi(x) u(t,x)\de x,
\end{equation}
which implies that $u_t$ is absolutely continuous with respect to the Lebesgue measure and
$$
\|u_t\|_{\mathcal{M}}=\|u(t,\cdot)\|_{L^1},\qquad \mbox{for a.e. } t\in (0,T).
$$
Note that is crucial that the space $C_0(\R^d)$ is separable, because the set of zero measure where \eqref{uguale limite L1} does not hold may depend on $\phi$. Finally, we can infer that $u\in L^\infty((0,T);L^1(\R^d))$ and therefore together with the weak convergence in $L^1((0,T)\times\R^d)$ we can conclude that
$$
u^\delta\weaktos u,\qquad\mbox{in }L^\infty((0,T);L^1(\R^d)),
$$
and this concludes the proof.
\end{proof}

A special sub-class of distributional solutions is given by \emph{parabolic solutions}, which are weakly differentiable in the space variable. This notion of solution is natural for vector fields and diffusion matrices possessing enough integrability in the space variable. 

\begin{definition}
	Assume that 
	\begin{equation}\label{eq:ass_para_sol_FP}
		\b\in L^1((0,T); L^2(\R^d;\R^d)), \qquad \a\in L^2((0,T); L^2(\R^d; \R^{d\times d})), \qquad u_0\in L^2(\R^d).
	\end{equation}
	A function $u\in L^{\infty}((0,T);L^2(\R^d))$ is a {\em parabolic solution} to \eqref{eq:fp} if it is a distributional solution to \eqref{eq:fp} and furthermore $u\in L^2((0,T);H^1(\R^d))$.
\end{definition}

We refer to the space $L^2((0,T);H^1(\R^d))$ as the \emph{parabolic class}. We recall that the Fokker-Planck equation \eqref{eq:fp} can be formally rewritten in divergence-form as
\begin{equation}\tag{FP-div}\label{eq:fp-div}
\begin{cases}
\partial_t u + \dive(\tilde{\b} u) -\frac{1}{2} \sum_{i,j}\partial_{i}(a_{ij}\partial_j u)=0 &  \text{ in } (0,T) \times \R^d, \\
u\vert_{t=0}=u_0 & \text{ in } \R^d,
\end{cases}
\end{equation}
by defining the vector field $\tilde{\b}:=\b-\frac{1}{2}\sum_j\partial_j a_{ij}$. The following result holds.
\begin{prop}\label{prop:equivalence}
Assume that $\a$ satisfies \ref{assu:dive_a_and_b}. Then, $u$ is a parabolic solution of \eqref{eq:fp} if and only if $u$ is a parabolic solution of \eqref{eq:divFP} with vector field $\tilde\b$.
\end{prop}
\begin{proof}
The proof is rather straightforward and relies on a simple integration by parts which makes use of the regularity assumption $u\in L^2((0,T);H^1(\R^d))$. In fact one can show that
\begin{align*}
\int_0^T\int_{\R^d}\bigg(u\b\cdot\nabla\varphi+\frac{1}{2}\sum_{ij}a_{ij}u\partial_{ij}\varphi\bigg) \de x \de t&=\int_0^T\int_{\R^d}\bigg(u\b\cdot\nabla\varphi-\frac{1}{2}\sum_{ij}(\partial_j a_{ij}u+a_{ij}\partial _j u)\partial_{i}\varphi\bigg) \de x \de t\\
&=\int_0^T\int_{\R^d}\bigg(u\tilde\b\cdot\nabla\varphi-\frac{1}{2}\sum_{ij}a_{ij}\partial _j u\partial_{i}\varphi\bigg) \de x \de t,
\end{align*}
where in the last line we used the definition of $\tilde \b$, i.e. 
$
\tilde \b=\b-\frac{1}{2}\sum_j\partial_j a_{ij}.
$
\end{proof}
In some situations it is convenient to use the formulation \eqref{eq:fp-div} since,  roughly speaking, we have to take into account one derivative less on $a_{ij}$. We will return on this point below. Under our assumptions, there exists at least one solution in the parabolic class: the result is the following.

\begin{proposition}[Existence of parabolic solutions]\label{prop:existence_parabolic_class_2}
	Let $\a,\b$ and $u_0$ be as in \eqref{eq:ass_para_sol_FP}, let \ref{assu:dive_a_and_b}, \ref{assu:elliptic} hold and assume that $(\dive \tilde \b)^- \in L^1((0,T);L^\infty(\R^d))$. Then there exists at least one parabolic solution to \eqref{eq:fp}.
\end{proposition}
\begin{proof}
The proof follows from the same regularization procedure developed in Proposition \ref{prop:existence_weak_sol_2}: we consider the unique solution $u^\delta$ of the regularized problem \eqref{eq:ad_bae}. Then, we define the function $\beta(s)=s^2/2$ and integrating in time the inequality \eqref{est:matr-1} we get
\begin{equation}\label{est:esistenza-parabolic}
\|u^\delta(t,\cdot)\|_{L^2}^2+\alpha\int_0^t\int_{\R^d} |\nabla u^\delta(s,x)|^2\de x\de s\leq \|u^\delta_0\|_{L^2}^2-\int_0^t\int_{\R^d}\dive\tilde{\b}^\delta(s,x)|u^\delta(s,x)|^2\de x\de s.
\end{equation}
From Proposition \ref{prop:existence_weak_sol_2} we know that 
$$
\|u^\delta(t,\cdot)\|_{L^2}^2\leq\|u^\delta_0\|_{L^2}^2 e^{\int_0^T\|(\dive\tilde\b)^-(s,\cdot)\|_{\infty}\de s},
$$ 
and together with \eqref{est:esistenza-parabolic} we obtain that
\begin{align}
    \int_0^T\int_{\R^d}|\nabla u^\delta(t,x)|^2\de x \de t&\leq \frac{1}{\alpha}\left(\|u_0\|_{L^2}^2+\|(\dive\tilde\b)^-\|_{L^1L^\infty}\|u^\delta\|_{L^\infty L^2}^2\right)\nonumber\\
    &\leq \frac{1}{\alpha}\left(\|u_0\|_{L^2}^2+\|(\dive\tilde\b)^-\|_{L^1L^\infty}\de s\|u_0\|_{L^2}^2 e^{\int_0^T\|(\dive\tilde\b)^-(s,\cdot)\|_{\infty}\de s}\right).
\end{align}
Thus, we have obtained that
$$
\|u^\delta\|_{L^\infty L^2}+\|\nabla u^\delta\|_{L^2L^2}\leq C(T,\alpha,\|u_0\|_{L^2},\|(\dive\tilde\b)^-\|_{L^1L^\infty}),
$$
and, since the equation is linear, the conclusion follows from a standard compactness argument.
\end{proof}

We can now state the uniqueness theorem for solutions in the parabolic class. The proof easily follows from Theorem \ref{thm:esistenza unicita fp div} and Proposition \ref{prop:equivalence}, thus we omit the details.

\begin{thm}[Uniqueness of parabolic solutions]\label{thm:uniqueness_weak_parabolic_FP} 
Let $u_0\in L^2\cap L^q(\R^d)$ be a given initial datum and assume that
\begin{itemize}
    \item[(i)] $\b\in L^2((0,T); L^p(\R^d;\R^d))$ with $\frac{1}{p}+\frac{1}{q}= \frac{1}{2}$;
    \item[(ii)] $\tilde\b$ satisfies the growth condition $$\frac{\tilde\b}{1+|x|}\in L^1((0,T);L^1(\R^d;\R^d))+L^1((0,T);L^\infty(\R^d;\R^d));$$
    \item[(iii)] $\a$ satisfies \ref{assu:bounded}, \ref{assu:dive_a_and_b} and \ref{assu:elliptic};
    \item[(iv)]
    $
    (\dive\tilde \b)^-\in L^1((0,T);L^\infty(\R^d)).
    $
\end{itemize}
Then, there exists a unique parabolic solution $u$ to \eqref{eq:fp}.
\end{thm}

\begin{rem}
Note that if $p=\infty$ and $q=2$ we recover the existence and uniqueness of parabolic solutions to \eqref{eq:fp} in \cite[Theorem 4.3]{F08}.
\end{rem}

\subsection{Further comments and remarks}
In this section we collect some observations concerning the uniqueness theorem, taking into account the different forms of the equation that one can consider. We point out that in \cite{LBL_CPDE} there is a nice comparison of the different form of the equations, showing how commutator estimates allow to transfer well-posed results from one equation to another.

\begin{itemize}

\item If we consider the Fokker-Planck equation in transport form, i.e.
\begin{equation}
    \partial_t u+\b \cdot \nabla u-\frac{1}{2}\sum_{ij}\partial_{i}(a_{ij}\partial_j u)=0,
\end{equation}
we can drop the assumption on the $L^p-L^q$ integrability of $\b$ and $u$ in Theorem \ref{thm:esistenza unicita fp div}. This means that we can consider $\b\in L^2((0,T);L^2(\R^d;\R^d))$ and $u_0\in L^2(\R^d)$, then we use the convergence of the commutator $r_1^\delta$ in Lemma \ref{lem:conv_comm} to prove the uniqueness of parabolic solutions. The same observation holds true if instead we consider the Fokker-Planck equation \eqref{eq:fp} and we assume that $\dive\b\in L^1L^\infty$.


\item We can relax the assumption on $\sum_{j=1}^d \partial_j a_{ij}\in L^\infty((0,T)\times\R^d;\R^d)$ in Proposition \ref{prop:existence_weak_sol_2} and Proposition \ref{prop:existence_parabolic_class_2}. Note that in \cite{F08} the author considers the case of \eqref{eq:fp} with a bounded vector field $\b$, thus he needs \ref{assu:dive_a_and_b} to use the formulation in divergence form with a bounded vector field $\tilde\b$. Indeed, the results in Section 3 hold true if one consider
$$
\tilde\b:=\b-\frac{1}{2}\sum_j\partial_{j}a_{ij}\in L^2((0,T);L^2\cap L^p(\R^d;\R^d)), 
$$
and the growth condition
$$
\tilde \b\in L^1((0,T);L^1(\R^d;\R^d))+L^1((0,T);L^\infty(\R^d;\R^d)).
$$

\item Another form of the equation that could be considered is the following
\begin{equation}\label{eq:2 derivate su u}
\partial_t u+\dive(\b u)-\frac{1}{2}\sum_{ij}a_{ij}\partial_{ij} u=0.
\end{equation}
Roughly speaking, if we test the equation with a test function $\varphi\in C^\infty_c([0,T)\times \R^d)$ and we integrate by parts we get that
\begin{equation}
\int_{0}^T\int_{\R^d}u\big(\partial_t\varphi+\b\cdot\nabla\varphi\big)-\frac{1}{2}\sum_{ij}\partial_j(a_{ij}\varphi)\partial_i u\,\de x\de t+\int_{\R^d}u_0\varphi(0,\cdot)\de x=0.
\end{equation}
Thus, we can define weak solutions as in Definition \ref{def:weak-solution} requiring $\a, \sum_j\partial_j a_{ij}\in L^2_{t,x}$. 
Then, if one assumes \ref{assu:bounded}, \ref{assu:dive_a_and_b}, \ref{assu:elliptic}, and $(\dive\tilde\b)^-\in L^1((0,T);L^\infty(\R^d))$ the existence and uniqueness of the weak solution to \eqref{eq:2 derivate su u} can be proved similarly to Theorem \ref{thm:esistenza unicita fp div}. 
\end{itemize}

\section{The Regularity Theorems} 
In this section we present some regularity results on distributional solutions of the Fokker-Planck equation. These then will be the key to achieving uniqueness in the class of distributional solutions in several contexts.

\subsection{Local regularity Theorem and uniqueness}
Below we prove a regularity result which guarantees that distributional solutions of \eqref{eq:fp} in the class $L^\infty((0,T);L^q(\R^d))$ are actually in $L^2((0,T);H^1_\mathrm{loc}(\R^d))$, provided that we have enough regularity on $\a$ and integrability $\b$.
\begin{thm}\label{thm:regolarita_FP}
Let $p,q \ge 1$ such that $\sfrac{1}{p}+\sfrac{1}{q}\leq \sfrac{1}{2}$. Assume that $\a\in L^2((0,T);W^{1,p}(\R^d;\R^{d \times d}))$ and $\b\in L^2((0,T); L^p(\R^d;\R^d))$ satisfy \ref{assu:dive_a_and_b} and \ref{assu:elliptic} and $(\dive\tilde\b)^-\in L^1((0,T);L^\infty(\R^d))$. Let $u\in L^\infty((0,T);L^q(\R^d))$ be a distributional solution to \eqref{eq:fp}, then $u\in L^2((0,T);H^1_\mathrm{loc}(\R^d))$.
\end{thm}

\begin{proof} To commence, we observe that $\sfrac{1}{p}+\sfrac{1}{q}\leq \sfrac{1}{2}$ clearly implies that both $p,q \ge 2$ and then any $u\in L^\infty((0,T);L^q(\R^d))$ belongs to $L^\infty((0,T);L^2_\mathrm{loc}(\R^d))$. The strategy is to prove that $\nabla u\in L^2((0,T); L^2_\mathrm{loc}(\R^d;\R^d))$ and this will be achieved exhibiting an approximating sequence $(u^\delta)_{\delta}$ enjoying uniform bounds on $\nabla u^\delta$: in turn, this estimate will be obtained as a consequence of Lemma \ref{lem:conv_comm2} and Lemma \ref{lem:conv_comm4}. 
	
Let $(\rho^\delta)_{\delta}$ be a standard family of mollifiers, then we know that the function $u^\delta:=u*\rho^\delta$ solves the equation \eqref{eq:divFP-e+comm} with $\tilde \b$ instead of $\b$. Consider a smooth function $\varphi$ such that 
$$
\varphi\geq 0,\hspace{0.3cm}\mathrm{supp} \ \varphi\subset B_2,\hspace{0.3cm}\varphi= 1 \mbox{ in }B_1,
$$
and define $\varphi_R(x)=\varphi\left(x/R\right)$. We multiply the equation in  \eqref{eq:divFP-e+comm} by $u^\delta\varphi_R$ and we integrate over $(0,t)\times\R^d$. Then, integration by parts together with \ref{assu:elliptic} gives that
\begin{align*}
\int_{\R^d}|u^\delta(t,x)|^2\varphi_R\de x&+\alpha\int_0^t\int_{\R^d}|\nabla u^\delta|^2\varphi_R\de x\de s
\leq \int_{\R^d}|u^\delta_0|^2\varphi_R\de x+\int_0^t\int_{\R^d}|u^\delta|^2\tilde\b\cdot\nabla\varphi_R \de x\de s\\
&-\int_0^t\int_{\R^d}|u^\delta|^2\dive \tilde \b\,\varphi_R \de x\de s+\sum_{ij}\int_0^t\int_{\R^d}\partial_ja_{ij}\frac{|u^\delta|^2}{2}\partial_i\varphi_R\de x\de s\\
&+\sum_{ij}\int_0^t\int_{\R^d}a_{ij}\frac{|u^\delta|^2}{2}\partial_j\partial_i\varphi_R\de x\de s+2\int_0^t\int_{\R^d}(r^\delta+s^\delta)u^\delta\varphi_R\de x\de s.
\end{align*}
We now estimate separately the terms on the right hand side (the constant may change from line to line and is allowed to depend on $T, \|\varphi\|_{C^2}$, $\|(\dive \tilde \b)^-\|_{L^1L^\infty}$, $\|u\|_{L^\infty L^q}$, $\|u_0\|_{L^q}$, $\|\partial_j a_{ij}\|_\infty$):
\begin{itemize}
\item $ \displaystyle\int_{\R^d}|u_0^\delta|^2\varphi_R\de x\leq \|u_0\|_{L^q}^2\|\varphi_R\|_{L^{\frac{q}{q-2}}}\leq C \|u_0\|_{L^q}^2 R^{d\frac{q-2}{q}}$;
\item if $s$ is such that $2/q+1/p+1/s=1$, then we have
\begin{align*}
\int_0^T\int_{\R^d}|u^\delta|^2 \tilde\b\cdot\nabla\varphi_R \de x \de t& \leq \|u^\delta\|_{L^\infty L^q}^2\|\tilde \b\|_{L^1L^p}\|\nabla\varphi_R\|_{L^s}\\
&\leq C\|u\|_{L^\infty L^q}^2\|\tilde\b\|_{L^1L^p}\|\nabla\varphi_R\|_{L^s}\\
&\leq C R^{\frac{d}{s}-1}\|u\|_{L^\infty L^q}^2\|\tilde\b\|_{L^1L^p};
\end{align*}
\item since $\varphi_R$ is positive, we can use the bound on the negative part of $\dive\tilde \b$ to obtain that
\begin{align*}
-\int_0^T\int_{\R^d}|u^\delta|^2\dive \tilde \b\,\varphi_R \de x\de s&\leq \int_0^T\|(\dive \tilde \b)^-(s,\cdot)\|_\infty\int_{\R^d}|u^\delta|^2\,\varphi_R \de x\de s\\
&\leq \|(\dive \tilde \b)^-\|_{L^1L^\infty}\|u^\delta\|_{L^\infty L^q}^2\|\varphi_R\|_{L^{\frac{q}{q-2}}}\\
&\leq \|(\dive \tilde \b)^-\|_{L^1L^\infty}\|u\|_{L^\infty L^q}^2 R^{d\frac{q-2}{q}};
\end{align*}
\item we use \ref{assu:dive_a_and_b} to obtain that 
\begin{align*}
\sum_{ij}\int_0^T\int_{\R^d}\partial_ja_{ij}\frac{|u^\delta|^2}{2}\partial_i\varphi_R\de x\de s&=\sum_i\int_0^T\int_{\R^d}\left(\sum_j\partial_ja_{ij}\right)\frac{|u^\delta|^2}{2}\partial_i\varphi_R\de x\de s\\
&\leq C\left\|\sum_j\partial_ja_{ij}\right\|_{L^\infty}\|u^\delta\|_{L^\infty L^q}^2\|\nabla\varphi_R\|_{L^{\frac{q}{q-2}}}\\
&\leq C\left\|\sum_j\partial_ja_{ij}\right\|_{L^\infty}\|u\|_{L^\infty L^q}^2 R^{d\frac{q-2}{q}-1};
\end{align*}
\item if $s$ is such that $2/q+1/p+1/s=1$, then we have
\begin{align*}
\int_0^T\int_{\R^d}a_{ij}\frac{|u^\delta|^2}{2}\partial_i\partial_j\varphi_R\de x\de t\leq C R^{\frac{d}{s}-2}\|u\|_{L^\infty L^q}^2\|a_{ij}\|_{L^1L^p}.
\end{align*}
\end{itemize}
For the commutator terms, first of all notice that $u^\delta \varphi_R\in H^1_0(B_{2R})$. Then, we have that
\begin{itemize}
    \item the transport commutator
\begin{align*}
\int_0^T\int_{\R^d}r^\delta u^\delta \varphi_R \de x \de t&\leq C \|r^\delta\|_{L^2 H^{-1}(B_{2R})} \|u^\delta\varphi_R\|_{L^2H^1(B_{2R})}\\
&\leq C  \|r^\delta\|_{L^2 H^{-1}(B_{2R})}\left( \|u^\delta\varphi_R\|_{L^\infty L^2(B_{2R})}+\|\nabla (u^\delta\varphi_R)\|_{L^2L^2(B_{2R})}\right);
\end{align*}
\item the diffusion commutator
\begin{align*}
\int_0^T\int_{\R^d}s^\delta u^\delta \varphi_R \de x \de t&\leq C \|s^\delta\|_{L^2 H^{-1}(B_{2R})} \|u^\delta\varphi_R\|_{L^2H^1(B_{2R})}\\
&= C  \|s^\delta\|_{L^2 H^{-1}(B_{2R})}\left( \|u^\delta \varphi_R\|_{L^\infty L^2(B_{2R})}+\|\nabla( u^\delta\varphi_R)\|_{L^2L^2(B_{2R})}\right).
\end{align*}
\end{itemize}
We now use Young inequality to obtain the bound
$$
\int_0^T\int_{\R^d}(r^\delta+s^\delta) u^\delta \varphi_R \de x \de t \leq C_\alpha (\|r^\delta\|^2_{L^2 H^{-1}(B_{2R})}+\|s^\delta\|^2_{L^2 H^{-1}(B_{2R})})+\frac{\alpha}{2} \|u^\delta\varphi_R\|_{L^2H^1(B_{2R})}^2.
$$
Then, $0\leq\varphi_R\leq 1$ implies $\varphi_R^2 \le \varphi_R$, thus
\begin{align*}
\|u^\delta\varphi_R\|_{L^2H^1(B_{2R})}^2\leq \int_0^T\int_{B_{2R}}|u^\delta|^2\varphi_R\de x \de s+\int_0^T\int_{B_{2R}}|\nabla u^\delta|^2\varphi_R\de x \de s+\int_0^T\int_{B_{2R}}|u^\delta|^2|\nabla\varphi_R|^2\de x \de s.
\end{align*}

The first and the last summand can be estimated similarly as above and gives respectively 
\[
\int_0^T\int_{B_{2R}}|u^\delta|^2\varphi_R\de x \de s \le C \|u\|_{L^\infty L^q}^2 R^{d\frac{q-2}{q}},
\]

\[
\int_0^T\int_{B_{2R}}|u^\delta|^2|\nabla\varphi_R|^2\de x \de s \le C \|u\|^2_{L^\infty L^q} R^{d\frac{q-2}{q}-2}.
\]
All in all, 
\begin{align*}
\alpha\int_0^t\int_{\R^d}|\nabla u^\delta|^2\varphi_R\de x\de s  & \le  C \left(R^{d\frac{q-2}{q}}+R^{\frac{d}{s}-1}+R^{d\frac{q-2}{q}-1}\right) + C_\alpha (\|r^\delta\|^2_{L^2 H^{-1}(B_{2R})}+\|s^\delta\|^2_{L^2 H^{-1}(B_{2R})}) \\ 
& \qquad + \frac{\alpha}{2} \int_0^T\int_{B_{2R}}|\nabla u^\delta|^2\varphi_R\de x \de s .
\end{align*}
Passing to the supremum in $t$ in the left hand side we obtain that  
\begin{align*}
\frac{\alpha}{2} \int_0^T \int_{B_R} |\nabla u^\delta|^2\de x \de s  &\le \frac{\alpha}{2}\int_0^T\int_{\R^d}|\nabla u^\delta|^2\varphi_R\de x\de s \\
&\le C \left(R^{d\frac{q-2}{q}}+R^{\frac{d}{s}-1}+R^{d\frac{q-2}{q}-1}\right) + C_\alpha (\|r^\delta\|^2_{L^2 H^{-1}(B_{2R})}+\|s^\delta\|^2_{L^2 H^{-1}(B_{2R})}).
\end{align*}
Finally, the commutators $r^\delta, s^\delta$ go to $0$ in $L^2H^{-1}_\mathrm{loc}$ by Lemma \ref{lem:conv_comm3} and Lemma \ref{lem:conv_comm4}, thus they are uniformly bounded for $\delta$ small enough (possibly depending on $R$), i.e.
$$
\|r^\delta\|^2_{L^2 H^{-1}(B_{2R})}+\|s^\delta\|^2_{L^2 H^{-1}(B_{2R})}\leq \bar C(R),
$$
for some constant $\bar C(R)$ which may diverge as $R\to \infty$.
Therefore, we have 
\begin{equation}\label{eq:stima_locale} 
\|\nabla u^\delta\|^2_{L^2L^2(B_R)}\leq C \left(R^{d\frac{q-2}{q}}+R^{\frac{d}{s}-1}+R^{d\frac{q-2}{q}-1}\right)+\bar C(R),
\end{equation}
which is finite uniformly in $\delta$ for any $R>0$. This proves that $\nabla u$ exists and belongs to $L^2((0,T);L^2(B_R))$.
\end{proof}

\begin{rem}
Notice that in the proof of Theorem \ref{thm:regolarita_FP} we do not use the bound
\begin{equation}\label{remark stima a priori}
\sup_{t\in (0,T)}\|u(t,\cdot)\|_{L^q}\leq \|u_0\|_{L^q},
\end{equation}
since $u$ is any given distributional solution and a priori we do not know if the energy bound \eqref{remark stima a priori} holds. However, a posteriori this is true under the assumptions of Theorem \ref{thm:regolarita_FP}: in fact, it can be easily proved using the regularity of the solution and the convergence of the commutators $r^\delta,s^\delta$ to $0$ in $L^1_{t,x,\mathrm{loc}}$. 
\end{rem}

\begin{rem}
In addition to regularity, Theorem \ref{thm:regolarita_FP} proves the existence of solutions in the class $L^2((0,T);H^1_\mathrm{loc}(\R^d))$. It would be interesting to prove the existence of solutions in this class with a more direct argument, instead of relying on regularity.
\end{rem}

Combining Theorem \ref{thm:regolarita_FP} and Theorem \ref{thm:uniqueness_weak_parabolic_FP}, we obtain the result which is transverse to that of \cite{F08}.
\begin{thm}
Let $\b\in L^\infty((0,T); L^\infty(\R^d;\R^d))$ and $\a\in L^\infty((0,T);W^{1,\infty}(\R^d;\R^{d\times d}))$ satisfy \ref{assu:dive_a_and_b}, \ref{ass:dive b}, and \ref{assu:elliptic}. Then there exists at most one distributional solution $u\in L^2((0,T)\times\R^d)$.
\end{thm}
\begin{proof}
Note that if we consider a distributional solution $u\in L^2((0,T);L^q(\R^d))$, arguing as in Theorem \ref{thm:regolarita_FP} it is easy to show that $u\in L^\infty((0,T);L^q(\R^d))$ if one assume that $\tilde \b \in L^\infty_tL^p_x$, $\a\in L^\infty_t W^{1,p}_x$, and $(\dive\tilde \b)^-\in L^\infty_{t,x}$ (the estimates depend on $\|u\|_{L^2 L^q}$). Thus, under our assumptions, if we set $p=\infty,q=2$, the solution $u$ belongs to $L^\infty((0,T);L^2(\R^d))$. Then, it is enough to repeat the proof of Theorem \ref{thm:regolarita_FP} noticing that the estimate \eqref{eq:stima_locale} is now uniform in $R$. In fact, by Remark \ref{rem:global_comm_h-1} the commutators converge in $L^2((0,T);H^{-1}(\R^d))$, thus the constant $\bar C$ is uniformly bounded in $R$. Moreover, if we substitute $q=2$ in the powers of $R$ in  \eqref{eq:stima_locale}, we obtain that
$$
\limsup_{R\to\infty} \|\nabla u\|_{L^2L^2(B_R)} < \infty.
$$
This implies the global information $\nabla u\in L^2((0,T);L^2(\R^d;\R^d))$. Then, we can conclude that $u\in L^\infty((0,T);L^2(\R^d))\cap L^2((0,T);H^1(\R^d))$ and the uniqueness follows from Theorem \ref{thm:uniqueness_weak_parabolic_FP}.
\end{proof}

\subsection{Extensions with suitable growth conditions}
We now discuss the role of the growth assumptions on $\b,\a$ to recover further uniqueness and regularity results for the Fokker-Planck equation \eqref{eq:fp}. We show that under suitable growth assumption uniqueness holds in the {\em local} parabolic class $L^2((0,T);H^{1}_\mathrm{loc}(\R^d))$. In particular, weak solutions of \eqref{eq:fp-div} are also unique in $L^2((0,T);H^{1}_\mathrm{loc}(\R^d))$ as a consequence of Proposition \ref{prop:equivalence}. The theorem is the following.

\begin{thm}\label{thm:uniqueness-local}
    Let $u_0\in L^q\cap L^\infty(\R^d)$ be a given initial datum with $q>2$ and assume that $\b\in L^2((0,T);L^p(\R^d;\R^d))$ and $\a\in L^2((0,T);W^{1,p}(\R^d;\R^{d \times d}))$ with $1/p+1/q= 1/2$ satisfy \ref{assu:dive_a_and_b} and \ref{assu:elliptic}, and $(\dive\tilde \b)^-\in L^1((0,T);L^\infty(\R^d))$. Moreover, we assume the following growth conditions
\begin{equation*}
\frac{\b(t,x)}{1+|x|}, \frac{\partial_j a_{ij}(t,x)}{1+|x|}\in L^1((0,T);L^1(\R^d;\R^d))+L^1((0,T);L^\infty(\R^d;\R^d)),
\end{equation*}
\begin{equation*}
    \frac{\a(t,x)}{1+|x|^2}\in L^1((0,T);L^1(\R^d;\R^{d\times d}))+L^1((0,T);L^\infty(\R^d;\R^{d\times d})).
\end{equation*}
Then, there exists at most one distributional solution $u\in L^\infty((0,T);L^q(\R^d))$.
\end{thm}
\begin{proof}
The proof is similar to the one of Theorem \ref{thm:esistenza unicita fp div} so we just sketch it. Let $u\in L^\infty((0,T);L^q(\R^d))$ be a distributional solution of \eqref{eq:fp} whose existence is proved in Proposition \ref{prop:existence_weak_sol_2}. Let $\rho^\delta$ be a standard mollifier, we have that the function $u^\delta:=u*\rho^\delta$ satisfies the equation
\begin{equation}\label{eq:regolarizzata-teo-loc}
\begin{cases}
\partial_t u^\delta +\dive(\tilde \b u^\delta)-\frac{1}{2}\sum_{ij}\partial_{i}(a_{ij}\partial_j u^\delta)=r^\delta+s^\delta,\\
u^\delta(0,\cdot)=u_0*\rho_\delta.
\end{cases}
\end{equation}
Given our assumptions on $p$ and $q$,  Theorem \ref{thm:regolarita_FP} implies that the solution $u$ belongs to the space $L^2((0,T);H^{1}_\mathrm{loc}(\R^d))$. Thus, as a consequence of Lemma \ref{lem:conv_comm2} and Lemma \ref{lem:conv_comm4} we know that $r^\delta,s^\delta\to 0$ in $L^2((0,T);H^{-1}_\mathrm{loc}(\R^d))$. 
Consider a smooth function $\varphi$ such that 
$$
\varphi\geq 0,\hspace{0.3cm}\mathrm{supp} \ \varphi\subset B_2,\hspace{0.3cm}\varphi= 1 \mbox{ in }B_1,
$$
and define $\varphi_R(x)=\varphi\left(x/R\right)$. 
Let $\beta:\R\to \R$ be the function defined as 
$$
\beta(z)=|z|^q.
$$
The following argument can be made rigorous by considering a regularized version of the function $\beta$ as we did in the proof of Theorem \ref{thm:esistenza unicita fp div}. We skip this technical detail for simplicity of exposition.
We multiply the equation \eqref{eq:regolarizzata-teo-loc} by $\beta'(u^\delta)\varphi_R$ and with computations similar to those of the proof of the Theorem \ref{thm:esistenza unicita fp div} we get that
\begin{align*}
\int_0^t\int_{\R^d}\dive(\tilde\b u^\delta)\beta'(u^\delta)\varphi_R\de x\de s=&  -\int_0^t \int_{\R^d} u^\delta\tilde\b\cdot\nabla    u^\delta \beta''(u^\delta)\varphi_R\de x\de s \\ 
& - \int_0^t\int_{\R^d} u^\delta\beta'(u^\delta )\tilde\b\cdot\nabla \varphi_R\de x\de s \\
=& -(q-1)\int_0^t\int_{\R^d}\tilde\b\cdot\nabla\beta(u^\delta)\varphi_R\de x\de s\\ 
& - \int_0^t\int_{\R^d} u^\delta\beta'(u^\delta )\tilde\b\cdot\nabla \varphi_R\de x\de s \\
=&(q-1)\int_0^t\int_{\R^d}\dive\tilde\b\beta(u^\delta)\varphi_R\de x\de s\\ 
& +(q-1)\int_0^t\int_{\R^d}\beta(u^\delta)\tilde \b\cdot\nabla\varphi_R\de x \de s\\
&-\int_0^t\int_{\R^d} u^\delta\beta'(u^\delta )\tilde\b\cdot\nabla \varphi_R\de x\de s .
\end{align*}
On the other hand, for the term involving the diffusion we use the assumption on the divergence obtaining that
\begin{align*}
& \int_0^t\int_{\R^d}\partial_{i}(a_{ij}\partial_j u^\delta)\beta'(u^\delta)\varphi_R\de x\de s \\ 
= & -\int_0^t\int_{\R^d}a_{ij}\partial_j u^\delta\partial_iu^\delta\beta''(u^\delta)\varphi_R\de x \de s -\int_0^t\int_{\R^d}a_{ij}\partial_j u^\delta\beta'(u^\delta)\partial_i\varphi_R\de x \de s\\
=& -\int_0^t\int_{\R^d}a_{ij}\partial_j u^\delta\partial_iu^\delta\beta''(u^\delta)\varphi_R\de x \de s -\int_0^t\int_{\R^d}a_{ij}\partial_j \beta(u^\delta)\partial_i\varphi_R\de x \de s\\
=&-\int_0^t\int_{\R^d}a_{ij}\partial_j u^\delta\partial_iu^\delta\beta''(u^\delta)\varphi_R\de x \de s+\int_0^t\int_{\R^d}\partial_ja_{ij} \beta(u^\delta)\partial_i\varphi_R\de x \de s\\
&+ \int_0^t\int_{\R^d}a_{ij} \beta(u^\delta)\partial_j\partial_i\varphi_R\de x \de s.
\end{align*}
Then, we use the uniform ellipticity of $\a$ to obtain
\begin{equation}
0\leq \alpha\int_0^T\int_{\R^d}|\nabla u^\delta|^2\beta''(u^\delta)\varphi_R\de x \de s\leq \int_0^T\int_{\R^d}a_{ij}\partial_i u^\delta\partial_j u^\delta\beta''(u^\delta)\varphi_R\de x \de s.
\end{equation}
We finally have the following inequality
\begin{align*}
\int_{\R^d}\beta(u^\delta)\varphi_R\de x&
\leq -(q-1)\int_0^t \int_{\R^d} \dive \tilde\b \beta(u^\delta) \varphi_R\de x\de s \\ 
&+C\int_0^t\int_{\R^d}|\tilde\b| \beta(u^\delta )|\nabla \varphi_R|\de x\de s+\int_0^t\int_{\R^d} a_{ij}\beta(u^\delta)\partial_i\partial_j\varphi_R\de x \de s\\
&+\int_0^t\int_{\R^d} \partial_ja_{ij}\beta(u^\delta)\partial_i\varphi_R\de x\de s+\int_0^t\int_{\R^d}(r^\delta+s^\delta)\beta'(u^\delta)\varphi_R\de x\de s.
\end{align*}
Since $\beta'(u^\delta)\varphi_R\in H^1_0(B_{2R})$ we can let $\delta\to 0$ to obtain that 
\begin{align*}
\int_{\R^d}\beta(u)\varphi_R\de x&\leq -(q-1)\int_0^t\int_{\R^d}\dive(\tilde\b) \beta(u)\varphi_R\de x\de s+C\int_0^t\int_{\R^d}\beta(u)|\tilde \b ||\nabla\varphi_R| \de x \de s\\
&+\int_0^t\int_{\R^d} \beta(u)a_{ij}\partial_i\partial_j\varphi_R\de x \de s+\int_0^t\int_{\R^d} \beta(u)\partial_ja_{ij}\partial_i\varphi_R\de x \de s\\
&=I+...+IV.
\end{align*}
Then, we can estimate $I$ and $II$ as in the proof of Theorem \ref{thm:esistenza unicita fp div}, while we use the growth assumptions on $\a$ to estimate $III$ and $IV$ as follows
\begin{align*}
I&\lesssim \int_0^t\|(\dive\tilde\b(s,\cdot))^-\|_{\infty}\int_{\R^d}\beta(u)\varphi_R\de x \de s,\\
II&\lesssim \|u_0\|_\infty^q\int_0^T\int_{|x|>R}|\tilde\b_1(s,x)|\de x\de s+\int_0^T\|\tilde\b_2(s,\cdot)\|_\infty\int_{|x|>R}|u(s,x)|^q\de x \de s ,\\
III&\lesssim \|u_0\|_\infty^q\int_0^T\int_{|x|>R}|\a_1(s,x)|\de x\de s+\int_0^T\|\a_2(s,\cdot)\|_\infty\int_{|x|>R}|u(s,x)|^q\de x \de s,\\
IV&\lesssim \|u_0\|_\infty^q\int_0^T\int_{|x|>R}|\partial_j a^1_{ij}(s,x)|\de x\de s+\int_0^T\|\partial_j a_{ij}^2(s,\cdot)\|_\infty\int_{|x|>R}|u(s,x)|^q\de x \de s,
\end{align*}
where it is worth to point out that in the estimate for $III$  we used that $\partial_i\partial_j\varphi_R\leq\frac{C}{R^2}$ together with the growth assumption on $\a$. Moreover, all the implicit constants in the estimate depends on the $C^2$ norm of $\varphi$ only. Thus, by defining $F(R):=II+III+IV$, which vanishes as $R\to\infty$, we get the following
\begin{equation*}
\int_{\R^d}\beta(u)\varphi_R\de x\leq   \int_0^t\|(\dive\tilde\b(s,\cdot))^-\|_\infty\int_{\R^d}\beta(u)\varphi_R\de x \de s+F(R),
\end{equation*}
and finally, we apply Gronwall's inequality to obtain that
\begin{equation*}
\int_{\R^d}\beta(u)\varphi_R\de x  \leq F(R)\exp\left(\|(\dive\tilde\b)^-\|_{L^1L^\infty}\right),
\end{equation*}
and by letting $R\to\infty$ we obtain that $|u|=0$ for a.e. $(t,x)\in(0,T)\times\R^d$.
\end{proof}

If we assume slightly stronger growth conditions, there is a regime where the solution is parabolic instead of just locally parabolic. The result is the following. 

\begin{thm}\label{thm:growth_uniqueness}
Let $u\in L^\infty((0,T);L^2\cap L^q(\R^d))$ be a distributional solution of \eqref{eq:fp}, and let $\b\in L^2((0,T);L^p(\R^d;\R^d))$ and $\a\in L^2((0,T);W^{1,p}(\R^d;\R^{d\times d}))$ with $1/p+1/q= 1/2$ satisfy \ref{assu:dive_a_and_b} and \ref{assu:elliptic}, and $(\dive\tilde\b)^-\in L^1((0,T);L^\infty(\R^d))$. Moreover, assume that $2<q\leq \frac{2d}{d-2}$ together with the following growth conditions
$$
\frac{\b(t,x)}{1+|x|},\frac{\partial_j a_{ij}(t,x)}{1+|x|}\in L^1((0,T);L^{\frac{q}{q-2}}(\R^d;\R^d))+L^1((0,T);L^\infty(\R^d;\R^d)),
$$
$$
\frac{\a(t,x)}{1+|x|^2}\in L^1((0,T);L^{\frac{q}{q-2}}(\R^d;\R^{d\times d}))+L^1((0,T);L^\infty(\R^d;\R^{d\times d})).
$$

Then $u\in L^2((0,T);H^1(\R^d))$.
\end{thm}
\begin{proof}
We can proceed as in Theorem \ref{thm:regolarita_FP}: we consider the equation for $u^\delta$ and we multiply it by $2 u^\delta \varphi_R$ where $\varphi_R$ is defined as in Theorem \ref{thm:uniqueness-local} and we obtain the estimate 
\begin{align*}
\int_{\R^d}|u^\delta(t,x)|^2\varphi_R\de x&+\alpha\int_0^t\int_{\R^d}|\nabla u^\delta|^2\varphi_R\de x\de s
\leq \int_{\R^d}|u^\delta_0|^2\varphi_R\de x+\int_0^t\int_{\R^d}|u^\delta|^2\tilde\b\cdot\nabla\varphi_R \de x\de s\\
&-\int_0^t\int_{\R^d}|u^\delta|^2\dive \tilde \b\,\varphi_R \de x\de s+\sum_{ij}\int_0^t\int_{\R^d}\partial_ja_{ij}\frac{|u^\delta|^2}{2}\partial_i\varphi_R\de x\de s\\
&+\sum_{ij}\int_0^t\int_{\R^d}a_{ij}\frac{|u^\delta|^2}{2}\partial_j\partial_i\varphi_R\de x\de s+2\int_0^t\int_{\R^d}(r^\delta+s^\delta)u^\delta\varphi_R\de x\de s.
\end{align*}
The difference with respect to the proof of Theorem \ref{thm:regolarita_FP} is in the estimates of the second, fourth and fifth term on the right hand side above. For example, for the second term we use the growth assumption on $\b$ to get that
\begin{align*}
\int_0^T\int_{\R^d}|u^\delta|^2 \b\cdot\nabla\varphi_R \de x \de t&\leq C\int_0^T\int_{R<|x|<2R}|u^\delta|^2 \frac{|\b|}{1+|x|}\de x \de t\\
&\leq C\|u_0\|_{L^2}^2\left\|\frac{|\b_2|}{1+|x|}\right\|_{L^{\infty}}\\
&+ C\|u_0\|^2_{L^q}\underbrace{\int_0^T\left(\int_{|x|>R}\left(\frac{|\b_1|}{1+|x|}\right)^{\frac{q}{q-2}}\de x \right)^{\frac{q-2}{q}}\de t}_{(*):=f(R)}
\end{align*}
where $f(R)\to 0$ as $R\to\infty$. Similarly, we can estimate the other two terms using the growth assumptions on $\a$. Then, by using that the commutators are uniformly bounded in $\delta$, we proceed as in Theorem \ref{thm:regolarita_FP} and we obtain that
\begin{equation*}
\|\nabla u^\delta\|^2_{L^2L^2(B_R)}\leq C \left(1+F(R)+R^{d\frac{q-2}{q}-2}\right),
\end{equation*}
for some function $F(R)$ which vanishes as $R\to \infty$. Thus we have obtained that $\|\nabla u^\delta\|^2_{L^2L^2(B_R)}$ is uniformly bounded (in $\delta$ and $R$) and this concludes the proof.
\end{proof}

Combining Theorem \ref{thm:growth_uniqueness} with the uniqueness result for parabolic solutions Theorem \ref{thm:uniqueness_weak_parabolic_FP} we obtain the following: 

\begin{cor} 
Assume that the hypothesis of Theorem \ref{thm:growth_uniqueness} are fulfilled. In addition assume that $\a$ satisfies \ref{assu:bounded}.
Then, any distributional solution $u\in L^\infty((0,T);L^2\cap L^q(\R^d))$ of \eqref{eq:fp} is parabolic and therefore unique.
\end{cor}

\subsection*{Acknowledgements} 
During the preparation of this manuscript, G. Ciampa has been supported by the ERC Starting Grant 101039762 HamDyWWa. He is currently supported by INdAM-GNAMPA, and by the projects PRIN2020 ``Nonlinear evolution PDEs, fluid dynamics and transport equations: theoretical foundations and applications” and PRIN2022 ``Classical equations of compressible fluids mechanics: existence and properties of non-classical solutions''. G. Crippa is supported by SNF Project 212573 FLUTURA – Fluids, Turbulence, Advection. Views and opinions expressed are however those of the authors only and do not necessarily reflect those of the European Union or the European Research Council. Neither the European Union nor the granting authority can be held responsible for them.

\end{document}